\documentclass[12pt, reqno]{amsart}
\usepackage{amsmath, amsthm, amscd, amsfonts, amssymb, graphicx, color}
\usepackage[bookmarksnumbered, colorlinks, plainpages]{hyperref}
\hypersetup{colorlinks=true,linkcolor=red, anchorcolor=green, citecolor=cyan, urlcolor=red, filecolor=magenta, pdftoolbar=true}

\usepackage{enumerate}
\usepackage{mathrsfs}

\allowdisplaybreaks[4] 

\newtheorem{theorem}{Theorem}[section]
\newtheorem{lemma}[theorem]{Lemma}
\newtheorem{assumption}[theorem]{Assumption}
\newtheorem{proposition}[theorem]{Proposition}

\theoremstyle{definition}
\newtheorem{definition}[theorem]{Definition}

\theoremstyle{remark}
\newtheorem{remark}[theorem]{Remark}
\numberwithin{equation}{section}
\begin{document}

\setcounter{page}{1}

\title[Multilinear operators]{Multilinear operators
on Hardy spaces associated with ball quasi-Banach function spaces}

\author[J. Tan]{Jian Tan}

\address{School of Science, Nanjing University of Posts and Telecommunications, Nanjing 210023, China.}
\email{\textcolor[rgb]{0.00,0.00,0.84}{tj@njupt.edu.cn; tanjian89@126.com}}




\date{Received: xxxxxx; Revised: yyyyyy; Accepted: zzzzzz.}

\begin{abstract}
This paper establishes that multilinear Calder\'on--Zygmund operators and their maximal operators are bounded on Hardy spaces associated with ball quasi-Banach function spaces. 
Moreover, we also obtain the boundedness of multilinear pseudo-differential operators 
on local Hardy spaces associated with ball quasi-Banach function spaces.
Since these (local) Hardy type spaces encompass a wide range of classical (local) Hardy-type spaces
including weighted (local) Hardy spaces, variable (local) Hardy space, (local) Hardy--Morrey space, mixed-norm (local) Hardy space, (local) Hardy--Lorentz space and (local) Hardy--Orlicz spaces, the results presented in this paper are highly general and essentially improve the existing results.
\end{abstract} \maketitle

\section{Introduction}

Multilinear Calder\'on--Zygmund theory is a natural extension of its linear counterpart. The foundational work on the class of multilinear Calder\'on--Zygmund operators was initiated by Coifman and Meyer in \cite{CM} and subsequently studied systematically by Grafakos and Torres in \cite{GT}. They demonstrated that  for any $1<p_1,\dots,p_m<\infty$ these operators map products of Lebesgue spaces $L^{p_1} \times \cdots \times L^{p_m}$ into the corresponding Lebesgue space $L^p$, with the relationship $\frac{1}{p} = \frac{1}{p_1} + \cdots + \frac{1}{p_m}$.  

On the other hand, Hardy spaces have played a pivotal role in the modern harmonic analysis since the groundbreaking contributions to Hardy space theory by Stein and Weiss \cite{SW} as well as Fefferman and Stein\cite{FS}. The study of multilinear operators within the framework of Hardy space theory has garnered increasing attention from various researchers; see, for instance, \cite{GK,HL}.  
Recently, significant progress has been made in establishing the boundedness of multilinear Calder\'on--Zygmund operators from products of weighted or variable Hardy spaces to weighted or variable Hardy spaces; see \cite{CMN,Tan}. It is worth noting that Lebesgue spaces, weighted Lebesgue spaces, and variable Lebesgue spaces are examples of ball quasi-Banach function spaces, which frequently serve as ingredients of Hardy-type spaces. For more details, see for example \cite{CJY,Ho,SHYY,SYY,Tan24,YJY} and the references therein.   
Thus, a natural question arises:\par  
{\it Can we establish that multilinear Calder\'on--Zygmund operators are bounded on Hardy spaces built on ball quasi-Banach function spaces?}  

The main purpose of this paper is to address this question, focusing on
the boundedness of multilinear Calder\'on--Zygmund operators and their maximal operators on the generalized Hardy type spaces.
Meanwhile, we also consider the boundedness of multilinear pseudo-differential operators on the generalized local Hardy type spaces.

Roughly speaking, 
to prove our results for the generalized Hardy type spaces,
we only need to assume the following three conditions on $X$:
First of all, given a ball quasi-Banach function space $X$, there exists $r>1$ such that $X^r$ is a ball Banach function space and the Hardy--Littlewood maximal operator is bounded on the associate space $(X^{r})'$,
where the \emph{r\text{-}convexification} $X^{r}$ of $X$ is defined by setting
        \[
        X^{r}:=\{f\in\mathscr{M}(\mathbb R^{n})\colon\vert f\vert^{r}\in X\}
        \]
        equipped with the \emph{quasi-norm} $\|f\|_{X^{r}}:=\|\vert f\vert^{r}\|_{X}^{1/r}$ for any $f\in X^{r}$.
        Secondly, for some $\theta,\ s\in(0,1]$ and $\theta<s$, there exists a positive constant $C$ such that, for any $\{f_{j}\}_{j=1}^{\infty}\subset L^{1}_{loc}(\mathbb R^{n})$,
\begin{align*}
\left\|  \left\{  \sum_{j=1}^{\infty}\left[ \mathcal{M}^{(\theta)}(f_{j}) \right]^{s} \right\}^{\frac{1}{s}}  \right\|_{X}\leq C\left\|  \left\{  \sum_{j=1}^{\infty}\vert f_{j} \vert^{s} \right\}^{\frac{1}{s}} \right\|_{X},
\end{align*}
where the \emph{powered Hardy--Littlewood maximal operator} $\mathcal{M}^{(\theta)}$ is defined by setting
\[
\mathcal{M}^{(\theta)}(f)(x):=\{\mathcal{M}(\vert f \vert^{\theta})(x)\}^{\frac{1}{\theta}}
\] for any $f\in L^{1}_{loc}(\mathbb R^{n})$ and $x\in\mathbb R^{n}$.
Last but not least, the following H\"older's inequality holds true:
\begin{align*}
\left\|\prod_{k=1}^mf_k\right\|_{X^p}\lesssim \prod_{k=1}^m\|f_k\|_{X_k^{p_k}}
\end{align*}
for any $f_k\in X_k^{p_k}$, $k=1,2,\dots,m$.

The Hardy spaces built on ball quasi-Banach function spaces encompass a diverse array of Hardy-type spaces, such as weighted Hardy spaces, variable Hardy spaces, Hardy--Morrey spaces, mixed-norm Hardy spaces, Hardy--Lorentz spaces, and Hardy--Orlicz spaces. Consequently, the results presented in this paper have a wide range of applications and even when $X$ is the Morrey space, the mixed-norm Lebesgue space, the Lorentz space and the Orlicz spaces, all these results are completely new. 

Before we state our main results, we recall the definition of 
the multilinear Calder\'on--Zygmund operators in \cite{GT}.
We assume that $K(y_0, y_1, \ldots, y_m)$ is a function away from
the diagonal $y_0=y_1=\cdots=y_m$ in $(\mathbb R^n)^{m}$ which
satisfies the following estimates
\begin{equation}\label{s1e1}
|\partial_{y_0}^{\alpha_0}\partial_{y_1}^{\alpha_1}\cdots
\partial_{y_m}^{\alpha_m}K(y_0, y_1, \ldots, y_m)|
\le\frac{C}{(\sum_{k,l=0}^m|y_k-y_l|)^{nm+|\alpha|}}
\
\end{equation}
for all $|\alpha|\le N$, where $\alpha=(\alpha_0, \alpha_1, \ldots, \alpha_m)$
is an ordered set of m-tuples of nonnegative integers, $|\alpha|=
|\alpha_0|+|\alpha_1|+\cdots+|\alpha_m|$, where $\alpha_j$ is the
order of each multi-index $a_k$, and $N>\widetilde N$ and $\widetilde N$ is a large integer to be determined
later. We call such functions $K$ multilinear standard kernels, $k=1, 2, \ldots, m$.
We assume that $T$ is a weakly continuous multilinear operator defined on products
of test functions such that for some multilinear standard kernel $K$,
the integral representation below is valid
\begin{equation*}\label{s1e2}
T(f_1,\ldots, f_m)(x)=\int_{(\mathbb R^n)^m}K(x, y_1, \ldots, y_m)
f_1(y_1)\cdots f_m(y_m)dy_1\cdots dy_m,
\end{equation*}
whenever every $f_k$ is the smooth function with compact support and
$x\in(\cap_{k=1}^m\mbox{supp}f_k)^c$.
We call $T$ an multilinear Calder\'on--Zygmund operator if it is
associated to a multilinear standard kernel as above and has a
bounded extension from a product of $L^{q_1}(\mathbb R^n), \ldots, L^{q_m}(\mathbb R^n)$ spaces
into another $L^q(\mathbb R^n)$ space with $\frac{1}{q}=\frac{1}{q_1}+\cdots+\frac{1}{q_m}$ for some $1<q_1,\ldots,q_m<\infty$.

Our first theorem gives the boundedness of multilinear Calder\'on--Zygmund operators from products of the generalized Hardy spaces into corresponding ball quasi-Banach function spaces. For brevity we defer some technical definitions to Section 2.

\begin{theorem}\label{s1th1}
Given an integer $m\ge 1$, let $0<p, p_1, \ldots, p_m<\infty$,
let $X^{p_1}_1,$ $\ldots,$ $X^{p_m}_m$ be ball quasi-Banach function spaces satisfying Assumption \ref{ass2.7} with $0<\theta<s<1$ and Assumption \ref{ass2.8} with the same $s$ in (\ref{2.7}),
and let $X^p$ be the ball quasi-Banach function spaces satisfying Assumption \ref{ass2.7} with $0<\theta_p<s_p<1$ and Assumption \ref{ass2.8} with the same $s_p$ in (\ref{2.7}). 
Let 
$T$ be a multilinear Calder\'on--Zygmund operator associated to the kernel $K$ that satisfies (\ref{s1e1}).
Suppose that the following H\"older's inequality holds true:
\begin{align}\label{holder}
\left\|\prod_{k=1}^mf_k\right\|_{X^p}\lesssim \prod_{k=1}^m\|f_k\|_{X_k^{p_k}}
\end{align}
for any $f_k\in X_k^{p_k}$.
If $\theta\in [\frac{mn}{n+N+1},1)$, then 
$$T: H_{X_1^{p_1}}\times \cdots \times H_{X_m^{p_m}}\rightarrow X^p.$$
\end{theorem}

Our second result gives boundedness of multilinear Calder\'on--Zygmund operators into corresponding Hardy spaces.

\begin{theorem}\label{s1th2}
Given an integer $m\ge 1$, let $0<p, p_1, \ldots, p_m<\infty$,
let $X^p,$ $X^{p_1}_1,$ $\ldots,$ $X^{p_m}_m$ be ball quasi-Banach function spaces satisfying Assumption \ref{ass2.7} with $0<\theta<s<1$ and Assumption \ref{ass2.8} with the same $s$ in (\ref{2.7}) and let $X^p$ be the ball quasi-Banach function spaces satisfying Assumption \ref{ass2.7} with $0<\theta_p<s_p<1$ and Assumption \ref{ass2.8} with the same $s_p$ in (\ref{2.7}). 
Let 
$T$ be a multilinear Calder\'on--Zygmund operator associated to the kernel $K$ that satisfies (\ref{s1e1}).
Suppose that the H\"older's inequality (\ref{holder}) holds true. Assume further that
\begin{align}\label{s1c3}
\int_{(\mathbb R^n)^m}x^\alpha T(a_1,a_2,\cdots,a_m)(x)dx=0,
\end{align}
for all $|\alpha|\le \widetilde N$ and $(X,q,N)$-atoms $a_i$,
$i=1,2,\cdots,m$.
If $\theta\in [\frac{mn}{n+N+1}\vee \frac{mn}{N-\tilde N},1)$
and $\theta_p\in [\frac{n}{n+\widetilde N+1},1)$,
then 
$$T: H_{X_1^{p_1}}\times \cdots \times H_{X_m^{p_m}}\rightarrow H_{X^p}.$$
\end{theorem}

In this paper we also consider the maximal multilinear Calder\'on--Zygmund operator
defined by
$$
T_\ast(\vec{f})(x)=\sup_{\delta>0}|T_\delta(f_1,\ldots,f_m)(x)|,
$$
where $T_\delta$ are smooth truncations of $T$ given by
$$
T_\delta(f_1,\ldots,f_m)(x)=\int_{(\mathbb R^n)^m}K_\delta
(x,y_1,\ldots,y_m)f_1(y_1)\cdots f_m(y_m)dy_1\cdots dy_m,
$$
where $$K_\delta(x,y_1,\ldots,y_m)=
\phi(\sqrt{|x-y_1|^2+\cdots+|x-y_m|^2}/2\delta)K(x,y_1,\ldots,y_m)$$
and $\phi(x)$ is a smooth function on $\mathbb R^n$,
which vanishes if $|x|\le 1/4$ and is equal to $1$ if $|x|>1/2$.
It was proved in \cite{GT1} that the sublinear operator $T_\ast$ satisfies similar boundedness
estimates as $T$.
Repeating the similar argument to that of Theorem \ref{s1th1}, we immediately obtain the following result regarding $T_\ast$.

\begin{theorem}\label{s1th3}
Let all the notation be as in Theorem \ref{s1th1}
and let $T$ be a multilinear Calder\'on--Zygmund operator associated to the kernel $K$ that satisfies (\ref{s1e1}).
Suppose that the H\"older's inequality (\ref{holder}) holds true.
If $\theta\in [\frac{mn}{n+N+1},1)$, then 
$$T_\ast: H_{X_1^{p_1}}\times \cdots \times H_{X_m^{p_m}}\rightarrow X^p.$$
\end{theorem}

Lastly, we will establish the boundedness of multilinear pseudo-differential operators on local Hardy spaces associated with ball quasi-Banach function spaces.
For $f_k\in\mathcal S, k=1,\cdots,m,$ recall that the multilinear pseudo-differential operators in \cite{GT,TZ} of the form
$$
T_\sigma(f_1,\ldots, f_m)(x)=\int_{(\mathbb R^n)^m}\sigma(x,\xi_1,\ldots,\xi_m)
\prod_{k=1}^m\hat f_k(\xi_k)e^{2\pi ix\cdot(\sum_{j=1}^m\xi_k)}
d\xi_1\cdots d\xi_m,$$
where the symbol $\sigma(x,\xi_1,\ldots,\xi_m)\in MS_{\rho,\delta}^m$,
for $m\in \mathbb R$ and $\rho,\delta\in[0,1]$, that is,
$\sigma(x,\xi_1,\ldots,\xi_m)$ is smooth
and
$$
|\partial_x^\alpha\partial_{\xi_1}^{\beta_1}
\cdots\partial_{\xi_m}^{\beta_m}
\sigma(x,\xi_1,\ldots,\xi_m)|
\le
C(1+\sum_{k=1}^m|\xi_k|)^{m-\rho(\sum_{k=1}^m
|\beta_k|)+\sigma|\alpha|},
$$
for all multi-indices $\alpha$, $\beta_k$, $k=1,\cdots,m$.
We denote $\sigma(x,\xi_1,\ldots,\xi_m)\in MS_{\rho,\delta}^m$
for the symbol of multilinear pseudo-differential operators.
Now we state our last result of our paper as follows.
\begin{theorem}\label{s1th4}
Let $p, p_1, \ldots, p_m$ and $X^{p_1}_1,$ $\ldots,$ $X^{p_m}_m$ be as in Theorem \ref{s1th1}
and let the H\"older's inequality (\ref{holder}) hold true.
If $T_\sigma$ be a multilinear pseudo-differential operator with the symbol $\sigma\in MB_{1,0}^0$,
then 
$$T_\sigma: h_{X_1^{p_1}}\times \cdots \times h_{X_m^{p_m}}\rightarrow X^p.$$
\end{theorem}

The remainder of this paper is organized as follows.
In Section 2, we recall some basic concepts and known results on the ball quasi-Banach function spaces that we will use
in subsequent sections. 
Then in Sections 3 and 4 we
prove our main results with the help of the atomic characterizations
of Hardy spaces built on ball quasi-Banach function spaces.
In Section 5, we apply all the above main results to six concrete examples of ball
quasi-Banach function spaces,
namely, weighted Lebesgue spaces,
variable Lebesgue spaces, Morrey spaces, mixed-norm Lebesgue spaces, Lorentz spaces and Orlicz spaces, respectively.
Therefore, all the boundedness of the multilinear operators on
corresponding (local) Hardy spaces are obtained. 

Throughout this paper, $C$ or $c$ will denote a positive constant that may vary at each occurrence
but is independent to the essential variables,
$A\lesssim B$ means that there are constants
$C>0$ independent of the essential variables such that $A\leq CB$ and
 $A\sim B$ means that $A\lesssim B$ and $B\lesssim A$.
Given a measurable set $S\subset \mathbb{R}^n$, $|S|$ denotes the Lebesgue measure and $\chi_S$
means the characteristic function. For a cube $Q$, let $Q^\ast$ denote with the same center
and $2\sqrt{n}$ its side length, i.e. $l(Q^\ast)=2\sqrt{n}l(Q)$.
Similarly, denote that $Q^{\ast\ast}=4nQ$.
The symbols $\mathcal S$ and $\mathcal S'$ denote the class of
Schwartz functions and tempered functions, respectively.
As usual, for a function $\psi$ on $\mathbb R^n$
and $\psi_t(x)=t^{-n}\psi(t^{-1}x)$.
We also use the notation $k\wedge k'=\min\{k,k'\}$ and $k\vee k'=\max\{k,k'\}$.

\section{Preliminaries}\label{se2}
In this section, we present some notation and known results that will be used throughout the paper.
First, we recall the definitions of ball quasi-Banach function spaces and their related (local) Hardy spaces. For any $x\in\mathbb R^{n}$ and $r\in(0,\infty)$, let $B(x,r):=\{y\in\mathbb R^{n}\colon\vert x-y\vert<r\}$ and
\begin{equation}\label{eq2.1}
\mathbb B(\mathbb R^{n}):=\left\{B(x,r)\colon x\in\mathbb R^{n}\ {\rm and}\ r\in(0,\infty)\right\}.
\end{equation}\par
The concept of ball quasi-Banach function spaces on $\mathbb R^{n}$ is as follows. For more details, see \cite{SHYY}.
\begin{definition}\label{def2.1}
    Let $X\subset\mathscr{M}(\mathbb R^{n})$ be a quasi-normed linear space equipped with a quasi-norm $\| \cdot \|_{X}$ which makes sense for all measurable functions on $\mathbb R^{n}$. Then $X$ is called a \emph{ball\ quasi\text{-}Banach\ function\ space} on $\mathbb R^{n}$ if it satisfies
    \begin{enumerate}[\quad(i)]
    \item if $f\in\mathscr{M}(\mathbb R^{n})$, then $\|f\|_{X}=0$ implies that $f=0$ almost everywhere;
    \item if $f, g\in\mathscr{M}(\mathbb R^{n})$, then $|g|\leq|f|$ almost everywhere implies that $\|g\|_{X}\leq\|f\|_{X}$;
    \item if $\{f_{m}\}_{m\in\mathbb N}\subset\mathscr{M}(\mathbb R^{n})$ and $f\in\mathscr{M}(\mathbb R^{n})$, then $0\leq f_{m}\uparrow f$ almost everywhere as $m\to\infty$ implies that $\|f_{m}\|_{X}\uparrow\|f\|_{X}$ as $m\to
    \infty$;
    \item $B\in\mathbb B(\mathbb R^{n})$ implies that $\chi_{B}\in X$, where $\mathbb B(\mathbb R^{n})$ is the same as (\ref{eq2.1}).
    \end{enumerate}
    Moreover, a ball quasi-Banach function space $X$ is called a \emph{ball\ Banach\ function\ space} if it satisfies
    \begin{enumerate}
    \item[(v)] for any $f,g\in X$
    \[
    \|f+g\|_{X}\leq\|f\|_{X}+\|g\|_{X};
    \]
    \item[(vi)] for any ball $B\in\mathbb B(\mathbb R^{n})$, there exists a positive constant $C_{(B)}$, depending on $B$, such that, for any $f\in X$,
    \[
    \int_{B}\vert f(x)\vert dx\leq C_{(B)}\| f\|_{X}.
    \]
    \end{enumerate}
\end{definition}
The associate space $X^{\prime}$ of any given ball Banach function space $X$ is defined as follows. 

Now we recall the definition of Hardy spaces and local Hardy spaces associated with ball quasi-Banach function spaces in \cite{SHYY}.

\begin{definition}\label{hx}
Let $X$ be a ball quasi-Banach function space and let $f\in \mathcal{S'}$,
$\psi\in \mathcal S$
and $\psi_t(x)=t^{-n}\psi(t^{-1}x)$, $x\in \mathbb{R}^n$.
Denote by $\mathcal{M}_N$ the grand maximal operator given by
$$\mathcal{M}_Nf(x)= \sup\{|\psi_t\ast f(x)|: t>0,\psi \in \mathcal{F}_N\}$$ for any fixed large integer $N$,
where $$\mathcal{F}_N=\{\varphi \in \mathcal{S}:\int\varphi(x)dx=1,\sum_{|\alpha|\leq N}\sup(1+|x|)^N|\partial ^\alpha \varphi(x)|\leq 1\}.$$
The Hardy space ${H}_{X}$ is the set of all $f\in \mathcal{S}^\prime$, for which the quantity
$$\|f\|_{{H}_{X}}=\|\mathcal{M}_Nf\|_{X}<\infty.$$
\end{definition}

\begin{remark}\label{local}
Similarly, we denote that $\mathcal{M}_{loc}$ is the local grand maximal operator given by
$$\mathcal{M}_{loc}f(x)= \sup\{|\psi_t\ast f(x)|: 0<t<1,\psi \in \mathcal{F}_N\}$$ for any fixed large integer $N$.
The local Hardy space ${h}_{X}$ is the set of all $f\in \mathcal{S}^\prime$, for which the quantity
$$\|f\|_{{h}_{X}}=\|\mathcal{M}_{loc}f\|_{X}<\infty.$$
\end{remark}

\begin{definition}\label{def2.2}
    For any given ball Banach function space $X$, its \emph{associate space} (also called the \emph{K\"othe dual space}) $X^{\prime}$ is defined by setting
    \[
    X^{\prime}:=\{f\in\mathscr{M}(\mathbb R^{n})\colon\|f\|_{X^{\prime}}<\infty\},
    \]
    where, for any $f\in X^{\prime}$,
    \[
    \|f\|_{X^{\prime}}:=\sup\left\{ \|fg\|_{L^{1}}\colon g\in X,\ \|g\|_{X}=1 \right\},
    \]
    and $\|\cdot\|_{X^{\prime}}$ is called the \emph{associate norm} of $\|\cdot\|_{X}$.
\end{definition}

Recall that $X$ is said to have an \emph{absolutely continuous quasi-norm} if, for any $f\in X$ and any measurable subsets $\{E_{j}\}_{j\in\mathbb N}\subset\mathbb R^{n}$ with both $E_{j+1}\subset E_{j}$ for any $j\in\mathbb N$ and $\bigcap_{j\in\mathbb N}E_{j}=\emptyset$,  $ \|f\chi_{E_{j}} \|_{X}\downarrow0$ as $j\to\infty$.
Then we give the definition of the $(X,q,d)$-atom.
\begin{definition}\label{def2.6}
    Let $X$ be a ball quasi-Banach function space and $q\in(1,\infty]$. Assume that $d\in\mathbb Z_{+}$. Then a measurable function $a$ is called a \emph{$(X,q,d)$-atom} if 
    \begin{enumerate}[(i)]
        \item there exists a cube $Q\subset\mathbb R^{n}$ such that ${\rm supp}a\subset Q$;
        \item $\|a\|_{L^{q}}\leq\frac{\vert Q \vert^{1/q}}{\|\chi_{Q}\|_{X}}$;
        \item $\int_{\mathbb R^{n}}a(x)x^{\alpha}dx=0$ for any multi-index $\alpha\in\mathbb Z^{n}_{+}$ with $\vert\alpha\vert\leq d$.
    \end{enumerate}
\end{definition}

The definition of the local-$(X,q,d)$-atom is very similar.
In fact, apart from the third condition, namely the vanishing condition, the other two conditions remain the same. We only need to modify the third condition as follows:
If $\vert Q \vert<1$, then $\int_{\mathbb R^{n}}a(x)x^{\alpha}dx=0$ for any multi-index $\alpha\in\mathbb Z^{n}_{+}$ with $\vert\alpha\vert\leq d$.

Denote by $L^{1}_{loc}(\mathbb R^{n})$ the set of all locally integral functions on $\mathbb R^{n}$. Recall that the \emph{Hardy--Littlewood maximal operator} $\mathcal{M}$ is defined by setting, for any measurable function $f$ and any $x\in\mathbb R^{n}$,
\[
\mathcal{M}(f)(x)=\sup\limits_{x\in Q}\frac{1}{\vert Q\vert}\int_{Q}f(u)du.
\]
For any $\theta\in \left(0,\infty\right)$, the {\it powered Hardy--Littlewood maximal operator} $\mathcal M^{\left(\theta\right)}$ is defined by setting, for all $f\in L_{{\rm loc}}^{1}\left(\mathbb{R}^{n}\right)$ and $x\in \mathbb{R}^{n}$,
\begin{align}\label{eq:the powered Hardy--Littlewood maximal operator}
	\mathcal M^{\left(\theta\right)}\left(f\right)\left(x\right) := \left\{\mathcal M\left(\left|f\right|^{\theta}\right)\left(x\right) \right\}^{1/\theta}.
\end{align}

Moreover, we also need two basic assumptions on $X$ as follows. 
\begin{assumption}\label{ass2.7}
Let $X$ be a ball quasi-Banach function space. For some $\theta,\ s\in(0,1]$ and $\theta<s$, there exists a positive constant $C$ such that, for any $\{f_{j}\}_{j=1}^{\infty}\subset L^{1}_{loc}(\mathbb R^{n})$,
\begin{align}\label{2.7}
\left\|  \left\{  \sum_{j=1}^{\infty}\left[ \mathcal{M}^{(\theta)}(f_{j}) \right]^{s} \right\}^{\frac{1}{s}}  \right\|_{X}\leq C\left\|  \left\{  \sum_{j=1}^{\infty}\vert f_{j} \vert^{s} \right\}^{\frac{1}{s}} \right\|_{X}.
\end{align}
\end{assumption}

\begin{assumption}\label{ass2.8}  Let $X$ be a ball quasi-Banach function space.
Fix $q>1$. Suppose that $0<s<q$. For any $f\in(X^{1/s})^{\prime}$,
    \[
    \left\| \mathcal{M}^{((q/s)^{\prime})}(f) \right\|_{(X^{1/s})^{\prime}}\leq C\|f\|_{(X^{1/s})^{\prime}}.
    \]
\end{assumption}

\begin{lemma}\label{l2.7}
    Let $X$ be a ball quasi-Banach function space satisfying Assumption \ref{ass2.8}. For all sequences of cubes $\{Q_{j}\}_{j=1}^{\infty}$ and non-negative functions $\{g_{j}\}_{j=1}^{\infty}$, if~$ \sum_{j=1}^{\infty}\chi_{Q_{j}}g_{j}\in X$, then
    \[
    \left\|  \sum_{j=1}^{\infty}\chi_{Q_{j}}g_{j} \right\|_{X}\leq C\left\| \sum_{j=1}^{\infty}\left( \frac{1}{\vert Q_{j}\vert}\int_{Q_{j}}g_{j}^{q}(y)dy \right)^{\frac{1}{q}}\chi_{Q_{j}} \right\|_{X}.
    \]
\end{lemma}

\begin{theorem}\label{atom}
    Let $X$ be a ball quasi-Banach function space satisfying Assumptions~{\ref{ass2.7}} and {\ref{ass2.8}}. 
    Suppose that $1<q\leq\infty$ and $d\in\mathbb Z_{+}$. Given countable collections of cubes $\{Q_{j}\}_{j=1}^{\infty}$, of non-negative coefficients $\{\lambda_{j}\}_{j=1}^{\infty}$ and of the $(X,q,d)$-atoms $\{a_{j}\}_{j=1}^{\infty}$, if 
    \[
    \left\|  \sum_{j=1}^{\infty}\frac{\lambda_{j}\chi_{Q_{j}}}{\|\chi_{Q_{j}}\|_{X}} \right\|_{X}<\infty,
    \]
    then the series $f=\sum_{j=1}^{\infty}\lambda_{j}a_{j}$ converges in 
    $\mathcal S'(\mathbb R^n)$ satisfying
    \[
    \|f\|_{H_{X}}\leq C\left\| \sum_{j=1}^{\infty}\frac{\lambda_{j}\chi_{Q_{j}}}{\|\chi_{Q_{j}}\|_{X}} \right\|_{X}.
    \]
\end{theorem}

\begin{definition}\label{def4.4}
    Let $X$ be a ball quasi-Banach function space satisfying Assumptions~{\ref{ass2.7}} and {\ref{ass2.8}}. 
    Let $1<q\leq\infty$ and $d\in\mathbb Z_{+}$. The \emph{finite atomic Hardy space} $H_{fin}^{X,q,d}(\mathbb R^{n})$ associated with $X$, is defined by
    \[
    H_{fin}^{X,q,d}(\mathbb R^{n})=\left\{ f\in\mathcal{S}^{\prime}(\mathbb R^{n})\colon f=\sum_{j=1}^{M}\lambda_{j}a_{j} \right\},
    \]
    where $\{a_{j}\}_{j=1}^{M}$ are $(X,q,d)$-atoms satisfying
    \[
    \left\| \sum_{j=1}^{M} \frac{\lambda_{j}\chi_{Q_{j}}}{\|\chi_{Q_{j}}\|_{X}}   \right\|_{X}<\infty.
    \]
    Furthermore, the quasi-norm $\|\cdot\|_{H_{fin}^{X,q,d}}$ in $H_{fin}^{X,q,d}(\mathbb R^{n})$ is defined by setting, for any $f\in H_{fin}^{X,q,d}(\mathbb R^{n})$,
    \[
    \|f\|_{H_{fin}^{X,q,d}}:=\inf\left\{\left\| \sum_{j=1}^{M} \frac{\lambda_{j}\chi_{Q_{j}}}{\|\chi_{Q_{j}}\|_{X}}   \right\|_{X}\right\},
    \]
    where the infimum is taken over all finite decomposition of $f$.
\end{definition}

In what follows, the symbol $\mathcal C(\mathbb R^{n})$ is defined to be the set of all
continuous complex-valued functions on $\mathbb R^n$.
\begin{theorem}\label{finite}
    Let $X$ be a ball quasi-Banach function space satisfying Assumptions~{\ref{ass2.7}} and {\ref{ass2.8}}. 
Further assume that $X$ has an absolutely continuous quasi-norm. Fix $q\in(1,\infty)$ and $d\in\mathbb{Z}_{+}$. For $f\in H_{fin}^{X,q,d}(\mathbb R^{n})$,
    \[
    \|f\|_{H_{fin}^{X,q,d}}\sim\|f\|_{H_{X}}.
    \]
Moreover,
for $f\in H_{fin}^{X,\infty,d}(\mathbb R^{n})\cap \mathcal C(\mathbb R^{n})$,
    \[
    \|f\|_{H_{fin}^{X,\infty,d}}\sim\|f\|_{H_{X}}.
    \]
\end{theorem}

\begin{remark}\label{localhx}
We only need to replace the $(X,q,d)$-atoms in the aforementioned theorem with the local-$(X,q,d)$-atoms, and the theorem will still hold for the local Hardy space. For more details, see \cite{WYY,CT}.
\end{remark}

In order to prove the main results in the next chapters, we also recall some classical definitions about weights. Suppose that a weight $\omega$ is a non-negative, locally integrable function such that $0<\omega(x)<\infty$ for almost every $x\in\mathbb R^{n}$. It is said that $\omega$ is in the \emph{Muckenhoupt class} $A_{p}$ for $1<p<\infty$ if
\[
[\omega]_{A_{p}}=\sup\limits_{Q}\left(\frac{1}{Q}\int_{Q}\omega(x)dx\right)\left(\frac{1}{Q}\int_{Q}\omega(x)^{-\frac{1}{p-1}}dx\right)^{p-1}<\infty,
\]
where $Q$ is any cube in $\mathbb {R}^{n}$ and when $p=1$, a weight $\omega\in A_{1}$ if for almost everywhere $x\in\mathbb {R}^{n}$,
\[
\mathcal{M}\omega(x)\leq C\omega(x),
\]
Therefore, define the set
\[
A_{\infty}=\bigcup\limits_{1\leq p<\infty}A_{p}.
\]
Given a weight $\omega\in A_{\infty}$, define $$q_{\omega}=\inf\{q\geq1\colon \omega\in A_{q}\}.$$

\section{Proof of Theorems \ref{s1th1}, \ref{s1th3}  and \ref{s1th4}}

Before presenting the proof of Theorem \ref{s1th1}, we first provide its weaker version,
where we need the additional assumption that each ball Banach function space has an
absolutely continuous quasi-norm.

\begin{theorem}\label{s3th1}\quad
Given an integer $m\ge 1$, let $0<p, p_1, \ldots, p_m<\infty$ and
let $X^p,$ $X^{p_1}_1,$ $\ldots,$ $X^{p_m}_m$ be ball quasi-Banach function spaces satisfying let $X^p,$ $X^{p_1}_1,$ $\ldots,$ $X^{p_m}_m$ be ball quasi-Banach function spaces satisfying Assumption \ref{ass2.7} with $0<\theta<s<1$ and Assumption \ref{ass2.8} with the same $s$ in (\ref{2.7}) and let $X^p$ be the ball quasi-Banach function spaces satisfying Assumption \ref{ass2.7} with $0<\theta_p<s_p<1$ and Assumption \ref{ass2.8} with the same $s_p$ in (\ref{2.7}).
Let 
$T$ be a multilinear Calder\'on--Zygmund operator associated to the kernel $K$ that satisfies (\ref{s1e1}).
Suppose that the H\"older's inequality (\ref{holder}) holds true.
Further assume that $X, X_1,\dots X_m$ have the absolutely continuous quasi-norm.
If $\theta\in [\frac{mn}{n+N+1},1)$, then 
$$T: H_{X_1^{p_1}}\times \cdots \times H_{X_m^{p_m}}\rightarrow X^p.$$
\end{theorem}

\begin{proof}
Let $1\le k\le m$ and fix arbitrary functions $f_k\in H_{fin}^{X_k^{p_k},\infty,N}(\mathbb R^{n})$.
By Theorem \ref{finite}, $f_k$ admits an atomic decomposition:
\begin{equation*}
  f_k=\sum_{j_k=1}^{M}\lambda_{k,j_k}a_{k,j_k},
\end{equation*}
where  $\lambda_{k,j_k}\ge 0$ and 
$a_{k,j_k}$ are $(X^{p_k}_k, \infty, N)$-atoms that fulfill
$$
{\rm supp}(a_{k,j_k})\subset Q_{k,j_k},\quad
\|a_{k,j_k}\|_{L^{\infty}}\leq\frac{1}{\|\chi_{Q_{k,j_k}}\|_{X^{p_k}_k}},\quad
\int_{\mathbb R^{n}}a_{k,j_k}(x)x^{\alpha}dx=0
$$ 
for any multi-index $\vert\alpha\vert\leq N$,
and 
\begin{equation}\label{s-1}
  \left\| \sum_{j_k=1}^{M} \frac{\lambda_{k,j_k}\chi_{Q_{k,j_k}}}{\|\chi_{Q_{k,j_k}}\|_{X^{p_k}_k}}   \right\|_{X_k^{p_k}}
  \lesssim \|f_k\|_{{H}_{X_k^{p_k}}}.
\end{equation}

First we show that
\begin{align}\label{s-2}
\|T(f_1,\dots,f_m)\|_{X^p}\lesssim \prod_{k=1}^m\left\| \sum_{j_k=1}^{M} \frac{\lambda_{k,j_k}\chi_{Q_{k,j_k}}}{\|\chi_{Q_{k,j_k}}\|_{X^{p_k}_k}}   \right\|_{X_k^{p_k}}.
\end{align}

By the fact that $T$ is m-linear, for a. e. $x\in\mathbb R^n$ we write
\begin{align*}
T(f_1,\dots,f_m)(x)=\sum_{j_1}\cdots\sum_{j_m}\lambda_{1,j_1}\cdots
\lambda_{m,j_m}T(a_{1,j_1},\dots,a_{m,j_m})(x).
\end{align*}

Define $R_{j_1,\dots,j_m}$ to be the smallest cube among $Q^\ast_{1,j_1},\dots, Q^\ast_{m,j_m}$. We split $T(f_1,\dots,f_m)$ into two terms as follows:
\begin{align*}
&|T(f_1,\dots,f_m)|\\
=&\sum_{j_1}\cdots\sum_{j_m}\lambda_{1,j_1}\cdots
\lambda_{m,j_m}|T(a_{1,j_1},\dots,a_{m,j_m})|\chi_{R_{j_1,\dots,j_m}}\\
&+\sum_{j_1}\cdots\sum_{j_m}\lambda_{1,j_1}\cdots
\lambda_{m,j_m}|T(a_{1,j_1},\dots,a_{m,j_m})|\chi_{(R_{j_1,\dots,j_m})^c}\\
=&I+II.
\end{align*}

By \cite[Lemma 3.5]{CMN}, we can choose $1<q<\infty$ such that
\begin{align*}
&\left(\frac{1}{|R_{j_1,\dots,j_m}|}\int_{R_{j_1,\dots,j_m}}
|T(a_{1,j_1},\dots,a_{m,j_m})|^q(x)dx\right)^{\frac{1}{q}}\\
&\qquad\qquad\lesssim \prod_{k=1}^m\inf_{z\in R_{j_1,\dots,j_m}}
\frac{\mathcal M(\chi_{Q_{k,j_l}})(z)^{\frac{n+N+1}{mn}}}{\|\chi_{Q_{k,j_k}}\|_{X^{p_k}_k}}.
\end{align*}

From this, Lemma \ref{l2.7}, the condition (\ref{holder}), the fact that $\theta\in [\frac{mn}{n+N+1},1)$, Assumption \ref{ass2.7} and (\ref{s-1}),
we conclude that
\begin{align*}
\|I\|_{X^p}&\lesssim
\left\|\sum_{j_1,\dots,j_m}\left(\prod_{k=1}^m\lambda_{k,j_k}\right)
\left(\prod_{k=1}^m\inf_{z\in R_{j_1,\dots,j_m}}
\frac{\mathcal M(\chi_{Q_{k,j_k}})(z)^{\frac{n+N+1}{mn}}}{\|\chi_{Q_{k,j_k}}\|_{X^{p_k}_k}}\right)\chi_{R_{j_1,\dots,j_m}}
\right\|_{X^p}\\
&\lesssim
\left\|\prod_{k=1}^m\left(\sum_{j_k}\lambda_{k,j_k}
\frac{\mathcal M(\chi_{Q_{k,j_k}})^{\frac{n+N+1}{mn}}}{\|\chi_{Q_{k,j_k}}\|_{X^{p_k}_k}}
\right)\right\|_{X^p}\\
&\lesssim \prod_{k=1}^m\left\|\sum_{j_k}\lambda_{k,j_k}\frac{(\mathcal M(\chi_{Q_k,j_k}))^{\frac{n+N+1}{mn}}}
{\|\chi_{Q_{k,j_k}}\|_{X^{p_k}_k}}\right\|_{X_k^{p_k}}\\
&\lesssim \prod_{k=1}^m\left\|\sum_{j_k}\lambda_{k,j_k}
\frac{\mathcal M^{(\theta)}(\chi_{Q_k,j_k})}
{\|\chi_{Q_{k,j_k}}\|_{X^{p_k}_k}}\right\|_{X_k^{p_k}}\\
&\lesssim  \prod_{k=1}^m\left\| \sum_{j_k=1}^{\mathcal M} \frac{\lambda_{k,j_k}\chi_{Q_{k,j_k}}}{\|\chi_{Q_{k,j_k}}\|_{X^{p_k}_k}}   \right\|_{X_k^{p_k}}
  \lesssim \prod_{k=1}^m\|f_k\|_{{H}_{X_k^{p_k}}}.
\end{align*}

Now we estimate the term $II$. Repeating nearly identical argument in \cite[Lemma 3.6]{CMN}, we find that
\begin{align*}
|T(a_{1,j_1},\dots,a_{m,j_m})|\chi_{(R_{j_1,\dots,j_m})^c}\lesssim
\prod_{k=1}^m\frac{(\mathcal M(\chi_{Q_k,j_k}))^{\frac{n+N+1}{mn}}}{\|\chi_{Q_{k,j_k}}\|_{X^{p_k}_k}}
\end{align*}  
and then
\begin{align*}
II&\lesssim\sum_{j_1}\cdots\sum_{j_m}
\prod_{k=1}^m \lambda_{k,j_k}\frac{(\mathcal M(\chi_{Q_k,j_k}))^{\frac{n+N+1}{mn}}}
{\|\chi_{Q_{k,j_k}}\|_{X^{p_k}_k}}\\
&=\prod_{k=1}^m\left[\sum_{j_k}\lambda_{k,j_k}\frac{(\mathcal M(\chi_{Q_k,j_k}))^{\frac{n+N+1}{mn}}}
{\|\chi_{Q_{k,j_k}}\|_{X^{p_k}_k}}\right].
\end{align*}

Then by the condition (\ref{holder}), the fact that $\theta\in [\frac{mn}{n+N+1},1)$, Assumption \ref{ass2.7} and (\ref{s-1}), we conclude that
\begin{align*}
\|II\|_{X^p}
&\lesssim \prod_{k=1}^m\left\|\sum_{j_k}\lambda_{k,j_k}\frac{(\mathcal M(\chi_{Q_k,j_k}))^{\frac{n+N+1}{mn}}}
{\|\chi_{Q_{k,j_k}}\|_{X^{p_k}_k}}\right\|_{X_k^{p_k}}\\
&\lesssim \prod_{k=1}^m\left\|\sum_{j_k}\lambda_{k,j_k}
\frac{\mathcal M^{(\theta)}(\chi_{Q_k,j_k})}
{\|\chi_{Q_{k,j_k}}\|_{X^{p_k}_k}}\right\|_{X_k^{p_k}}\\
&\lesssim  \prod_{k=1}^m\left\| \sum_{j_k=1}^{M} \frac{\lambda_{k,j_k}\chi_{Q_{k,j_k}}}{\|\chi_{Q_{k,j_k}}\|_{X^{p_k}_k}}   \right\|_{X_k^{p_k}}
  \lesssim \prod_{k=1}^m\|f_k\|_{{H}_{X_k^{p_k}}}.
\end{align*}

Therefore we combine the estimates for $I$ and $II$, 
for any $f_k\in H_{fin}^{X_k^{p_k},\infty,N}(\mathbb R^{n})$,
we obtain that
\begin{align*}
\|T(f_1,\dots,f_m)\|_{X^p}
  \lesssim \prod_{k=1}^m\|f_k\|_{{H}_{X_k^{p_k}}}.
\end{align*}

If we combine the estimates for $I$ and $II$,
we prove this theorem.
\end{proof}

Based on Theorem \ref{s3th1}, we can relax the assumption that every ball quasi-Banach function space possesses an absolutely continuous quasi-norm. This relaxation extends the applicability of the theorem to function spaces such as Morrey spaces.
To give the proof of the main result, we also need the following very useful technical lemma
in \cite{CWYZ}.

\begin{lemma}\label{embed}
Let $X$ be a ball quasi-Banach function space satisfying Assumption
\ref{ass2.8} with $s\in (0,\infty)$ and $q\in (s,\infty)$. Then there exists an $\varepsilon \in (1-\frac{s}{q}, 1)$ such that $X$ continuously embeds into $L_\omega^s(\mathbb{R}^n)$ with 
\[
\omega := [\mathcal{M}(\mathbf{1}_{B(\vec{0}, 1)})]^\varepsilon.
\]
\end{lemma}

\begin{proof}[{\bf Proof of Theorem \ref{s1th1}}]
Given $f_k\in {H}_{X_k^{p_k}}$, by Theorem \ref{atom}, for any $1<q<\infty$
we find that there exist a sequence $\{a_{k,j_k}\}$ of $(X^{p_k}_k, q, N)$-atom 
and a sequence $\{\lambda_{k,j_k}\}$ of non-negative numbers,
such that
\begin{equation*}
  f_k=\sum_{j_k=1}^{\infty}\lambda_{k,j_k}a_{k,j_k}\quad\quad\mbox{in}\quad \mathcal S',
\end{equation*}
and 
\begin{equation}\label{s-1}
  \left\| \sum_{j_k=1}^{\infty} \frac{\lambda_{k,j_k}\chi_{Q_{k,j_k}}}{\|\chi_{Q_{k,j_k}}\|_{X^{p_k}_k}}   \right\|_{X_k^{p_k}}
  \lesssim \|f_k\|_{{H}_{X_k^{p_k}}}.
\end{equation}

Choose some $s, s_1,\dots, s_m\in (0,\infty)$ fulfilling 
$\frac{1}{s}=\frac{1}{s_1}+\cdots+\frac{1}{s_m}$. 
By Lemma \ref{embed},  there exists an $\varepsilon \in (0, 1)$ such that $X^{p}$, $X^{p_1}_1$,  $\cdots$, $X^{p_m}_m$ continuously embeds into $L_\omega^{s}$, $L_\omega^{s_1}$, $\cdots$, $L_\omega^{s_m}$,
respectively.

We denote $H_\omega^{r}$ the weighted Hardy space as in Definition \ref{hx} with $X$ replaced by $L_\omega^{r}$ for any $0<r<\infty$ and denote that $\widetilde\lambda_j:=\lambda_{k,j_k}\frac{\|\chi_{Q_{k,j_k}}\|_{L_\omega^{s_k}}}{\|\chi_{Q_{k,j_k}}\|_{X^{p_k}_k}}$ and $\widetilde a_j:=\frac{\|\chi_{Q_{k,j_k}}\|_{X^{p_k}_k}}{\|\chi_{Q_{k,j_k}}\|_{L_\omega^{s_k}}}a_{k,j_k}$.
Observe that $\omega\in A_1$ and that
\begin{align*}
  &\left\| \sum_{j_k=1}^{\infty}\frac{\widetilde \lambda_{k,j_k}\chi_{Q_{k,j_k}}}{\|\chi_{Q_{k,j_k}}\|_{L_\omega^{s_k}}}   \right\|_{L_\omega^{s_k}}
  =\left\| \sum_{j_k=1}^{\infty} \frac{\lambda_{k,j_k}\chi_{Q_{k,j_k}}}{\|\chi_{Q_{k,j_k}}\|_{X^{p_k}_k}}   \right\|_{L_\omega^{s_k}}\\
  &\leq \left\| \sum_{j_k=1}^{\infty} \frac{\lambda_{k,j_k}\chi_{Q_{k,j_k}}}{\|\chi_{Q_{k,j_k}}\|_{X^{p_k}_k}}   \right\|_{X_k^{p_k}}
  \lesssim \|f_k\|_{{H}_{X_k^{p_k}}}<\infty.
\end{align*}

Then we have
\begin{align*}
f_k=\sum_{j_k=1}^{\infty}\widetilde\lambda_{k,j_k}\widetilde a_{k,j_k}=
\sum_{j_k=1}^{\infty}\lambda_{k,j_k}a_{k,j_k}
\end{align*}
in $\mathcal S'$ and $H^{s_k}_{\omega}$,
where $\widetilde a_{k,j_k}$ is
a $(L_\omega^{s_k}, q, N)$-atom.
Therefore, from the known result that $T$ is bounded from $H_\omega^{s_1}\times\cdots\times H_\omega^{s_m}$ to $H_\omega^{s}$ we know that
\begin{align*}
T(f_1,\dots,f_m)(x)=\sum_{j_1}^\infty\cdots\sum_{j_m}^\infty\lambda_{1,j_1}\cdots
\lambda_{m,j_m}T(a_{1,j_1},\dots,a_{m,j_m})(x)
\end{align*}
in $\mathcal S'$ and $H^{s_k}_{\omega}$.
By using this and following the arguments in the proof of Theorem \ref{s3th1}, we obtain the desired result. This completes the proof of Theorem \ref{s1th1}.
\end{proof}

\begin{remark}
From the proofs of Theorem \ref{s3th1} and Theorem \ref{s1th1}, we find that, in fact,
the condition $\theta\in [\frac{mn}{n+N+1},1)$ is only required for ball quasi-Banach function spaces $X^{p_1}_1,$ $\ldots,$ $X^{p_m}_m$ satisfying Assumption \ref{ass2.7} with $0<\theta<s<1$, and not for the ball quasi-Banach function space $X^p$. 
\end{remark}

The proof of Theorem \ref{s1th3} is very similar to that of Theorem \ref{s1th1}. For brevity, we only show the differences.

\begin{proof}[Proof of Theorem \ref{s1th3}]
Let all the notation as in Theorem \ref{s1th1}. To end it, we only need to prove that
\begin{align*}
\|T_\ast(f_1,\dots,f_m)\|_{X^p}\lesssim \prod_{k=1}^m\left\| \sum_{j_k=1}^{M} \frac{\lambda_{k,j_k}\chi_{Q_{k,j_k}}}{\|\chi_{Q_{k,j_k}}\|_{X^{p_k}_k}}   \right\|_{X_k^{p_k}}.
\end{align*}

From the proof of \cite[Theorem 1.4]{Tan1} or \cite[Theorem 1.3]{WWX}, observe that
\begin{align*}
&\left(\frac{1}{|R_{j_1,\dots,j_m}|}\int_{R_{j_1,\dots,j_m}}
|T_\ast(a_{1,j_1},\dots,a_{m,j_m})|^q(x)dx\right)^{\frac{1}{q}}\\
&\qquad\qquad\lesssim \prod_{k=1}^m\inf_{z\in R_{j_1,\dots,j_m}}
\frac{\mathcal M(\chi_{Q_{k,j_l}})(z)^{\frac{n+N+1}{mn}}}{\|\chi_{Q_{k,j_k}}\|_{X^{p_k}_k}}.
\end{align*}
and that
\begin{align*}
|T_\ast(a_{1,j_1},\dots,a_{m,j_m})|\chi_{(R_{j_1,\dots,j_m})^c}\lesssim
\prod_{k=1}^m\frac{(\mathcal M(\chi_{Q_k,j_k}))^{\frac{n+N+1}{mn}}}{\|\chi_{Q_{k,j_k}}\|_{X^{p_k}_k}}.
\end{align*}
 The rest of the proof of Theorem \ref{s1th3} is identical to that of Theorem \ref{s1th1}. Thus, we have completed this theorem.
\end{proof}

At the end of this chapter, we give the proof of Theorem \ref{s1th4}.

\begin{proof}[Proof of Theorem \ref{s1th4}]
Let $1\le k\le m$ and fix arbitrary functions $f_k\in h_{fin}^{X_k^{p_k},q,N}(\mathbb R^{n})$.
By Remark \ref{localhx}, for any $1<q<\infty$,
$f_j$ admits an atomic decomposition:
\begin{equation*}
  f_k=\sum_{j_k=1}^{M}\lambda_{k,j_k}a_{k,j_k},
\end{equation*}
where  $\lambda_{k,j_k}\ge 0$ and 
$a_{k,j_k}$ are local $(X^{p_k}_k, q, N)$-atoms
and 
\begin{equation*}
  \left\| \sum_{j_k=1}^{M} \frac{\lambda_{k,j_k}\chi_{Q_{k,j_k}}}{\|\chi_{Q_{k,j_k}}\|_{X^{p_k}_k}}   \right\|_{X_k^{p_k}}
  \lesssim \|f_k\|_{{h}_{X_k^{p_k}}}.
\end{equation*}

To show this theorem, we only need to prove that
\begin{align*}
\|T_{\sigma}(f_1,\dots,f_m)\|_{X^p}\lesssim \prod_{k=1}^m\left\| \sum_{j_k=1}^{M} \frac{\lambda_{k,j_k}\chi_{Q_{k,j_k}}}{\|\chi_{Q_{k,j_k}}\|_{X^{p_k}_k}}   \right\|_{X_k^{p_k}}.
\end{align*}

For a. e. $x\in\mathbb R^n$ we rewrite
\begin{align*}
T_{\sigma}(f_1,\dots,f_m)(x)=\sum_{j_1}\cdots\sum_{j_m}\lambda_{1,j_1}\cdots
\lambda_{m,j_m}T_\sigma(a_{1,j_1},\dots,a_{m,j_m})(x).
\end{align*}

Fix atoms $a_{1,k_1},\ldots,a_{m,j_m}$
supported in cubes $Q_{1,j_1},\ldots,Q_{m,j_m}$ respectively.
Assume that $Q^\ast_{1,k_1}\cap\cdots\cap Q^\ast_{m,k_m}\neq\varnothing,$
otherwise there is nothing to prove. Without loss of generality,
assume that $Q_{1,j_1}$ has the
smallest size among all these cubes. Since
$Q^\ast_{1,j_1}\cap\cdots\cap Q^\ast_{m,j_m}\neq\varnothing,$
we can pick a cube $R_{j_1,\ldots,j_m}$ such that
\begin{align*}
Q^\ast_{1,j_1}\cap\cdots\cap Q^\ast_{m,j_m}
\subset R_{j_1,\ldots,j_m}\subset R^\ast_{j_1,\ldots,j_m}
\subset Q^{\ast\ast}_{1,j_1}\cap\cdots\cap Q^{\ast\ast}_{m,j_m}
\end{align*}
and
$$|Q_{1,j_1}|\le C|R_{j_1,\ldots,j_m}|.$$
We split $T_\sigma(f_1,\dots,f_m)$ into two terms as follows:
\begin{align*}
&|T_\sigma(f_1,\dots,f_m)|\\
=&\sum_{j_1}\cdots\sum_{j_m}\lambda_{1,j_1}\cdots
\lambda_{m,j_m}|T_\sigma(a_{1,j_1},\dots,a_{m,j_m})|\chi_{R_{j_1,\dots,j_m}}\\
&+\sum_{j_1}\cdots\sum_{j_m}\lambda_{1,j_1}\cdots
\lambda_{m,j_m}|T_\sigma(a_{1,j_1},\dots,a_{m,j_m})|\chi_{(R_{j_1,\dots,j_m})^c}\\
=&I+II.
\end{align*}

To prove the first term, by applying
Lemma \ref{l2.7}, the condition (\ref{holder})
and Assumption \ref{ass2.7}, we get that
\begin{align*}
&\|I\|_{X^p}\\
&=
\left\|\sum_{j_1}\cdots\sum_{j_m}\prod_{k=1}^m|\lambda_{k,j_k}|
|T_\sigma(a_{1,j_1},\ldots,a_{m,j_m})|
\chi_{R_{j_1,\dots,j_m}}
\right\|_{X^p}\\
&\lesssim
\left\|\sum_{j_1}\cdots\sum_{j_m}\prod_{k=1}^m|\lambda_{k,j_k}|
\left|\frac{1}{|R_{j_1,\ldots,j_m}|}\int_{R_{j_1,\ldots,j_m}}
T_\sigma(a_{1,j_1},\ldots,a_{m,j_m})(y)dy\right|
\chi_{R_{j_1,\ldots,j_m}}\right\|_{X^p}\\
&\lesssim \left\|\sum_{j_1}\cdots\sum_{j_m}\prod_{k=1}^m|\lambda_{k,j_k}|
\prod_{k=1}^m
\frac{1}{\|\chi_{Q_{k,j_j}}\|_{X_k^{p_k}}}
\prod_{k=1}^m\chi_{Q^{\ast\ast}_{k,j_k}}\right\|_{X^{p}}\\
&\lesssim \prod_{k=1}^m\left\|\sum_{j_k}
\frac{\lambda_{k,j_k}\chi_{Q^{\ast\ast}_{k,j_k}}}
{\|\chi_{Q_{k,j_k}}\|_{X^{p_k}_k}}\right\|_{X_k^{p_k}}\\
&\lesssim  \prod_{k=1}^m\left\| \sum_{j_k} \frac{\lambda_{k,j_k}\chi_{Q_{k,j_k}}}{\|\chi_{Q_{k,j_k}}\|_{X^{p_k}_k}}   \right\|_{X_k^{p_k}}.
\end{align*}

Let $A$ be a nonempty
subset of $\{1,\ldots,m\}$, and we denote the cardinality of
$A$ by $|A|$, then $1\le|A|\le m$. 
Now we estimate the term $II$. Repeating the same argument as in \cite[Theorem 1.1]{TZ}, we conclude that
\begin{align*}
|T(a_{1,j_1},\dots,a_{m,j_m})|\chi_{(R_{j_1,\dots,j_m})^c}\lesssim
\prod_{k=1}^m\frac{(\mathcal M(\chi_{Q_k,j_k}))^{\frac{n+\frac{M+1}{|A|}}{n}}}{\|\chi_{Q_{k,j_k}}\|_{X^{p_k}_k}}
\end{align*}  
for all $M\ge 1$.
Then by 
the condition (\ref{holder}), the fact that $M>mn(\frac{1}{\theta}-1)-1$ and Assumption \ref{ass2.7}, we get that
\begin{align*}
\|II\|_{X^p}
&\lesssim \prod_{k=1}^m\left\|\sum_{j_k}\lambda_{k,j_k}
\frac{\mathcal M^{(\theta)}(\chi_{Q_k,j_k})}
{\|\chi_{Q_{k,j_k}}\|_{X^{p_k}_k}}\right\|_{X_k^{p_k}}\\
&\lesssim  \prod_{k=1}^m\left\| \sum_{j_k=1}^{M} \frac{\lambda_{k,j_k}\chi_{Q_{k,j_k}}}{\|\chi_{Q_{k,j_k}}\|_{X^{p_k}_k}}   \right\|_{X_k^{p_k}}.
\end{align*}

The rest of the proof is identical to that of Theorem \ref{s1th1}. Hence, we complete the proof of Theorem
\ref{s1th4}.
\end{proof}

 \section{Proof of Theorem \ref{s1th2}}
 
To establish Theorem \ref{s1th2}, it suffices to prove the following weaker version. The remainder of the proof closely parallels that of the previous section.
 
\begin{theorem}\label{s4th1}\quad
Given an integer $m\ge 1$, let $0<p, p_1, \ldots, p_m<\infty$ and
let $X^p,$ $X^{p_1}_1,$ $\ldots,$ $X^{p_m}_m$ be ball quasi-Banach function spaces satisfying Assumption \ref{ass2.7} with $0<\theta<s<1$ and Assumption \ref{ass2.8} with the same $s$ in (\ref{2.7}) and let $X^p$ be the ball quasi-Banach function spaces satisfying Assumption \ref{ass2.7} with $0<\theta_p<s_p<1$ and Assumption \ref{ass2.8} with the same $s_p$ in (\ref{2.7}).
Let 
$T$ be a multilinear Calder\'on--Zygmund operator associated to the kernel $K$ that satisfies (\ref{s1e1}).
Suppose that the H\"older's inequality (\ref{holder}) holds true. Assume further that
the condition (\ref{s1c3}) holds true and that $X, X_1,\dots X_m$ have the absolutely continuous quasi-norm.
If $\theta\in [\frac{mn}{n+N+1}\vee \frac{mn}{N-\tilde N},1)$
and $\theta_p\in [\frac{n}{n+\widetilde N+1},1)$, then 
$$T: H_{X_1^{p_1}}\times \cdots \times H_{X_m^{p_m}}\rightarrow H_{X^p}.$$
\end{theorem} 
 
 \begin{proof}
 
 The proof of Theorem \ref{s4th1} is similar to the one of Theorem \ref{s3th1}.
For brevity, we only show the difference.
Here we will need to estimate the norm of $M_\phi \circ T$, 
where $M_\phi$ is the non-tangential maximal operator
\[
M_\phi f(x) = \sup_{0<t<\infty} \sup_{|y-x|<t} |\phi_t \ast f(y)|,
\]
where $\phi \in C_0^\infty$ and $\operatorname{supp}(\phi) \subset B(0,1)$. 
From \cite[Section 3.1]{SHYY}, by the condition that $X^p$ is a ball quasi-Banach function space satisfying Assumptions \ref{ass2.7}, we know that $H_{X^p}$ can be characterized via the non-tangential maximal function $M_\phi$ 
with the norm
\[
\|f\|_{H_{X^p}} \sim \|M_\phi f\|_{X^p}.
\]

Let $1\le k\le m$ and fix arbitrary functions $f_k\in H_{fin}^{X_k^{p_k},\infty,N}(\mathbb R^{n})$.
We only need to show that
\begin{align}\label{s-3}
\|M_\phi T(f_1,\dots,f_m)\|_{X^p}\lesssim \prod_{k=1}^m\left\| \sum_{j_k=1}^{M} \frac{\lambda_{k,j_k}\chi_{Q_{k,j_k}}}{\|\chi_{Q_{k,j_k}}\|_{X^{p_k}_k}}   \right\|_{X_k^{p_k}}.
\end{align}
 
Since that $M_\phi\circ T$ is multi-sublinear, we write
\begin{align*}
&M_\phi T(f_1,\dots,f_m)(x)\\
\le&\sum_{j_1}\cdots\sum_{j_m}\lambda_{1,j_1}\cdots
\lambda_{m,j_m}|M_\phi T(a_{1,j_1},\dots,a_{m,j_m})|\chi_{R^\ast_{j_1,\dots,j_m}}\\
&+\sum_{j_1}\cdots\sum_{j_m}\lambda_{1,j_1}\cdots
\lambda_{m,j_m}|M_\phi T(a_{1,j_1},\dots,a_{m,j_m})|\chi_{(R^\ast_{j_1,\dots,j_m})^c}\\
=:&I_1+II_1,
\end{align*}
where $R^\ast_{j_1,\dots,j_m}$ to be the smallest cube among $Q^{\ast\ast}_{1,j_1},\dots, Q^{\ast\ast}_{m,j_m}$. 
 
Observe that $M_\phi$ is controlled pointwise by the Hardy–Littlewood maximal operator.
Repeating a very similar argument as in the proof of Theorem \ref{s3th1},
we conclude that
\begin{align*}
\|I_1\|_{X^p}\lesssim  \prod_{k=1}^m\left\| \sum_{j_k=1}^{M} \frac{\lambda_{k,j_k}\chi_{Q_{k,j_k}}}{\|\chi_{Q_{k,j_k}}\|_{X^{p_k}_k}}   \right\|_{X_k^{p_k}}.
\end{align*}
 
Suppose that $Q_1$ is such that $\ell(Q_1) = \min\{\ell(Q_k) : 1 \leq k \leq m\}$. 
By \cite[Lemma 5.1]{CMN}, we can get that
for all $x \notin Q_1^{**}$, we have
\begin{align*}
&M_\phi T(a_1, \ldots, a_m)\\
&\lesssim \prod_{k=1}^m \frac{\mathcal M(\chi_{Q_k})^{\frac{n+N+1}{mn}}}{\|\chi_{Q_{k,j_k}}\|_{X^{p_k}_k}} 
+ {\mathcal M(\chi_{Q_1})^{\frac{n+\widetilde N+1}{n}}} 
\prod_{k=1}^m \inf_{z \in Q_1} \frac{\mathcal M(\chi_{Q_k})(z)^{\frac{N-\widetilde N}{mn}}}{\|\chi_{Q_{k,j_k}}\|_{X^{p_k}_k}} \\
&=:I_2+II_2.
\end{align*}

We only need to estimate the term $II_2$. 
The estimate of $I_2$ is nearly identical to the one of $II$ in Theorem \ref{s3th1}
because of $\theta\in [\frac{mn}{n+N+1},1)$.
Then by Assumption \ref{ass2.7}, the condition (\ref{holder}) and the fact that $\theta_p\in [\frac{n}{n+\widetilde N+1},1)$ and that $\theta\in [\frac{mn}{N-\tilde N},1)$, we get that
\begin{align*}
&\|II_2\|_{X^p}\\
&= \left\|\sum_{j_1}\cdots\sum_{j_m}\lambda_{1,j_1}\cdots\lambda_{m,j_m}
{\mathcal M(\chi_{Q_1})^{\frac{n+\widetilde N+1}{n}}}
\prod_{k=1}^m \inf_{z \in Q_1} \frac{\mathcal M(\chi_{Q_k})(z)^{\frac{N-\widetilde N}{mn}}}{\|\chi_{Q_{k,j_k}}\|_{X^{p_k}_k}} \right\|_{X^{p}}\\
&\lesssim \left\|\sum_{j_1}\cdots\sum_{j_m}\lambda_{1,j_1}\cdots\lambda_{m,j_m}
\chi_{R^\ast_{j_1,\ldots,j_m}} 
\prod_{k=1}^m \inf_{z \in R^\ast_{j_1,\ldots,j_m}} \frac{\mathcal M(\chi_{Q_k})(z)^{\frac{N-\widetilde N}{mn}}}{\|\chi_{Q_{k,j_k}}\|_{X^{p_k}_k}} \right\|_{X^{p}}\\
&\lesssim \left\|\sum_{j_1}\cdots\sum_{j_m}\lambda_{1,j_1}\cdots\lambda_{m,j_m}
\prod_{k=1}^m \frac{\mathcal M(\chi_{Q_k})^{\frac{N-\widetilde N}{mn}}}{\|\chi_{Q_{k,j_k}}\|_{X^{p_k}_k}} \right\|_{X^{p}}\\
&\lesssim \prod_{k=1}^m\left\|\sum_{j_k}\lambda_{k,j_k}\frac{(\mathcal M(\chi_{Q_k,j_k}))^{\frac{N-\widetilde N}{mn}}}
{\|\chi_{Q_{k,j_k}}\|_{X^{p_k}_k}}\right\|_{X_k^{p_k}}
\lesssim  \prod_{k=1}^m\left\| \sum_{j_k=1}^{M} \frac{\lambda_{k,j_k}\chi_{Q_{k,j_k}}}{\|\chi_{Q_{k,j_k}}\|_{X^{p_k}_k}}   \right\|_{X_k^{p_k}}.
\end{align*}

Therefore, we have completed
the proof of Theorem \ref{s4th1}.
\end{proof}

\section{Applications}
In this section, we apply the main theorems to  several concrete examples of ball quasi-Banach function spaces, namely, weighted Lebesgue spaces (Subsection \ref{s5s1}),
variable Lebesgue spaces (Subsection \ref{s5s2}), Morrey spaces (Subsection \ref{s5s3}), mixed-norm Lebesgue spaces (Subsection \ref{s5s4}), Lorentz spaces (Subsection \ref{s5s5}) and Orlicz spaces (Subsection \ref{s5s6}).
To our best knowledge, our results on the boundedness of these multilinear operators and multi-sublinear operators on the (local) Hardy--Morrey spaces, mixed-norm (local) Hardy spaces, (local) Hardy--Lorentz spaces and (local) Hardy--Orlicz spaces are completely new. Given the generality of the results, further applications are both expected and foreseeable.

\subsection{Weighted (local) Hardy spaces}\label{s5s1}
From \cite[Subsection 7.1]{SHYY} we know that $L^{p}_{\omega}(\mathbb R^{n})$ with $p\in(0,\infty)$ and $\omega\in A_{\infty}(\mathbb R^{n})$ is a ball quasi-Banach function space. Moreover, 
by \cite[Remark 2.4(b)]{WYY} we know that Assumption~\ref{ass2.7} holds true when $\theta,\ s\in(0,1]$, $\theta<s$ and $X:=L^{p}_{\omega}(\mathbb R^{n})$ with $p\in(\theta,\infty)$ and $\omega\in A_{p/\theta}(\mathbb R^{n})$.
By \cite[Remark 2.7(b)]{WYY}, for any \( s \in (0, \min \{1, p\}) \), \( w \in A_{p/s} (\mathbb{R}^n) \) and \( q \in (\max\{1, p\}, \infty] \) large enough such that \( w^{1 - (p/s)'} \in A_{(p/s)' / (q/s)'} (\mathbb{R}^n) \), it holds true that \( (p/s)' / (q/s)' > 1 \) and
\[
\left[ (X^{1/s})' \right]^{1/(q/s)'} = L_{w^{1 - (p/s)'}}^{(p/s)' / (q/s)'} (\mathbb{R}^n),
\]
which, together with the boundedness of \( M \) on the weighted Lebesgue space, further implies that Assumption~\ref{ass2.8} holds true.

In this case, we have $X=L^{1}_{\bar \omega}(\mathbb R^{n})$, $X_1=L^{1}_{\omega_1}(\mathbb R^{n})$, $\dots,$
$X_m=L^{1}_{\omega_m}(\mathbb R^{n})$, where $\bar\omega=\prod_{k=1}^m\omega_k^{\frac{p}{p_k}}$ and
$$\frac{1}{p}=\sum_{k=1}^m\frac{1}{p_k},$$
for $k=1,\dots, m.$
Then by classical H\"older's inequality, we can get that for all $f_k\in L^{p_k}_{\omega_k}$, 
$$\left\|\prod_{k=1}^mf_k\right\|_{L^{p}_{\bar\omega}}\le \prod_{k=1}^m\|f_k\|_{L^{p_k}_{\omega_k}}.$$
From this and by Theorems \ref{s1th1}, \ref{s1th2}, \ref{s1th3} and \ref{s1th4}, it follows the
following conclusions on weighted (local) Hardy spaces.

\begin{theorem}\label{th51}\quad
Given an integer $m\ge 1$, $p_1,\dots,p_m\in (0,\infty)$, 
and $q_{w_k}$ is the critical index of ${\omega_k}\in A_\infty$, $1\le k\le m$, let 
$T$ be an multilinear Calder\'on--Zygmund operator associated to the kernel $K$ that satisfies (\ref{s1e1}).
If $\frac{mnq_{\omega_k}}{n+N+1}<p_k<\infty$, then $T$ extends to a bounded operator from
$H^{p_1}_{\omega_1}\times\cdots\times H^{p_m}_{\omega_m}$ into $L^{p}_{\bar \omega}$,
and $$
\frac{1}{p}=\frac{1}{p_1}+\cdots+\frac{1}{p_m}.
$$
\end{theorem}

If we impose the condition on $N$ such that
$$
N \geq \max \left\{\left\lfloor m n\left(\frac{q_{w_k}}{p_k}-1\right)\right\rfloor_{+}, 1 \leq k \leq m\right\}+(m-1) n,
$$
then by Theorem \ref{th51} for any $0<p_1,\dots,p_m<\infty$, $T$ is a bounded operator from
$H^{p_1}_{\omega_1}\times\cdots\times H^{p_m}_{\omega_m}$ into $L^{p}_{\bar \omega}$.

\begin{theorem}\label{th52}\quad
Let all the notation be the same as in Theorem \ref{th51}.
 Assume further that
\begin{align}\label{s1c3}
\int_{(\mathbb R^n)^m}x^\alpha T(a_1,a_2,\cdots,a_m)(x)dx=0,
\end{align}
for all $|\alpha|\le \widetilde N$ and $(L_{\omega_k}^{p_k},q,N)$-atoms $a_k$,
$k=1,2,\cdots,m$.
If $\frac{mnq_{\omega_k}}{n+N+1}\vee\frac{mnq_{\omega_k}}{N-\widetilde N}<p_k<\infty$
and $\frac{nq_{\omega_k}}{n+\widetilde N+1}<p<\infty$, then $T$ extends to a bounded operator from
$H^{p_1}_{\omega_1}\times\cdots\times H^{p_m}_{\omega_m}$ into $H^{p}_{\bar \omega}$.
\end{theorem}

Similarly, if we impose the condition on $\widetilde N, N$ in Theorem \ref{th52} such that
$$
\widetilde N\geq n\left(\frac{q_{\bar w}}{p}-1\right)
$$
and
$$
N \geq \widetilde N+\max \left\{\left\lfloor mn\left(\frac{q_{w_k}}{p_k}-1\right)\right\rfloor_{+}, 1 \leq k \leq m\right\}+mn,
$$
then for any $0<p_1,\dots,p_m<\infty$, we get that $T$ is a bounded operator from
$H^{p_1}_{\omega_1}\times\cdots\times H^{p_m}_{\omega_m}$ into $H^{p}_{\bar \omega}$.
The descriptions of the ranges of indices in all the theorems below are similar; therefore, we will not elaborate on them individually.
We also remark that Cruz-Uribe et al. have established the above boundedness of the multilinear Calder\'on--Zygmund operators on weighted Hardy spaces in \cite[Theorems 1.1 and 1.2]{CMN}. For more details, also see \cite{XY}. Meanwhile, as a corollary of Theorem \ref{s1th3}  we also get the following 
weighted Hardy spaces boundedness for the maximal operators of multilinear Calder\'on--Zygmund operators, which also have proved by Wen et al. in \cite{WWX}.

\begin{theorem}\label{th53}\quad
Let all the notation be as in Theorem \ref{th51}
and let $T$ be a multilinear Calder\'on--Zygmund operator associated to the kernel $K$ that satisfies (\ref{s1e1}).
If $\frac{mnq_{\omega_k}}{n+N+1}<p_k<\infty$, then $T_\ast$ extends to a bounded operator from
$H^{p_1}_{\omega_1}\times\cdots\times H^{p_m}_{\omega_m}$ into $L^{p}_{\bar \omega}$.
\end{theorem}

Denote $h_\omega^{r}$ the weighted Hardy space as in Remark \ref{local} with $X$ replaced by $L_\omega^{r}$ for any $0<r<\infty$. The following result is on the boundedness of multilinear pseudo-differential operators.

\begin{theorem}\label{th51-1}\quad
Given an integer $m\ge 1$, $p_1,\dots,p_m\in (0,\infty)$ such that $
\frac{1}{p}=\frac{1}{p_1}+\cdots+\frac{1}{p_m}
$ and ${\omega_k}\in A_\infty$, $1\le k\le m$, let 
$T_\sigma$ be a multilinear pseudo-differential operator with the symbol $\sigma\in MB_{1,0}^0$.
If $0<p_k<\infty$, then $T_\sigma$ extends to a bounded operator from
$h^{p_1}_{\omega_1}\times\cdots\times h^{p_m}_{\omega_m}$ into $L^{p}_{\bar \omega}$.
\end{theorem}

\subsection{Variable (local) Hardy spaces}\label{s5s2}
The variable exponent function spaces,
such as the variable Lebesgue spaces and the variable
Sobolev spaces, were studied by a substantial number
of researchers (see, for instance, \cite{CFMP,KR}).
For any Lebesgue measurable function $p(\cdot):
\mathbb R^n\rightarrow (0,\infty]$ and for any
measurable subset $E\subset \mathbb{R}^n$, we denote
$p^-(E)= \inf_{x\in E}p(x)$ and $p^+(E)= \sup_{x\in E}p(x).$
Especially, we denote $p^-=p^{-}(\mathbb{R}^n)$ and $p^+=p^{+}(\mathbb{R}^n)$.
Let $p(\cdot)$: $\mathbb{R}^n\rightarrow(0,\infty)$ be a measurable
function with $0<p^-\leq p^+ <\infty$ and $\mathcal{P}^0$
be the set of all these $p(\cdot)$.
Let $\mathcal{P}$ denote the set of all measurable functions
$p(\cdot):\mathbb{R}^n \rightarrow[1,\infty) $ such that
$1<p^-\leq p^+ <\infty.$

\begin{definition}\label{s1de1}\quad
Let $p(\cdot):\mathbb R^n\rightarrow (0,\infty]$
be a Lebesgue measurable function.
The {\it variable Lebesgue space} $L^{p(\cdot)}$ consisits of all
Lebesgue measurable functions $f$, for which the quantity
$\int_{\mathbb{R}^n}|\varepsilon f(x)|^{p(x)}dx$ is finite for some
$\varepsilon>0$ and
$$\|f\|_{L^{p(\cdot)}}=\inf{\left\{\lambda>0: \int_{\mathbb{R}^n}\left(\frac{|f(x)|}{\lambda}\right)^{p(x)}dx\leq 1 \right\}}.$$
\end{definition}
As a special case of the theory of Nakano and Luxemberg, we see that $L^{p(\cdot)}$
is a quasi-normed space. Especially, when $p^-\geq1$, $L^{p(\cdot)}$ is a Banach space. In the study of variable exponent function spaces it is common
to assume that the exponent function $p(\cdot)$ satisfies $LH$
condition.
We say that $p(\cdot)\in LH$, if $p(\cdot)$ satisfies

 $$|p(x)-p(y)|\leq \frac{C}{-\log(|x-y|)} ,\quad |x-y| \leq 1/2$$
and
 $$|p(x)-p(y)|\leq \frac{C}{\log(|x|+e)} ,\quad |y|\geq |x|.$$

 In fact, from \cite[Subsection 7.4]{SHYY} we know that whenever $p(\cdot)\in\mathcal{P}^{0}$, $L^{p(\cdot)}(\mathbb R^{n})$ is a ball quasi-Banach function space. Let $X:=L^{p(\cdot)}(\mathbb R^{n})$ with $p(\cdot)\in LH$. By \cite[Remark 2.4 (g)]{WYY}, we know that Assumption~\ref{ass2.7} holds true when $s\in(0,1]$ and $\theta\in(0,\min\{s,p_{-}\})$. Besides, let $X:=L^{p(\cdot)}(\mathbb R^{n})$ with $p(\cdot)\in LH$ and $0<p^{-}\leq p^{+}<\infty$. From \cite[Remark 2.7 (f)]{WYY}, Assumption~\ref{ass2.8} holds true when $s\in(0,\min\{1,p^{-}\})$ and $q\in(\max\{1,p^{+}\},\infty]$. Moreover,
the following generalized H\"{o}lder inequality on variable Lebesgue spaces
can be found in in \cite{CF, TLZ}.

\begin{proposition}
Given exponent function $p_k(\cdot)\in \mathcal{P}^0,$ define
$p(\cdot)\in \mathcal{P}^0$ by
$$\frac{1}{p(x)}=\sum_{k=1}^m\frac{1}{p_k(x)},$$
where $k=1,\dots, m.$
Then for all $f_k\in L^{p_k(\cdot)}(\mathbb R^{n})$ we have
$$\left\|\prod_{k=1}^mf_k\right\|_{L^{p(\cdot)}}\lesssim \prod_{k=1}^m\|f_k\|_{L^{p_k(\cdot)}}.$$
\end{proposition}

Therefore, Theorems~\ref{s1th1},\ \ref{s1th2}, \ref{s1th3} and \ref{s1th4} hold true in the variable exponents settings, where $X=L^{q(\cdot)}(\mathbb R^{n})$, $X_1=L^{q_1(\cdot)}(\mathbb R^{n})$, $\dots,$
$X_m=L^{q_m(\cdot)}(\mathbb R^{n})$ and $p=p_1=\cdots=p_m=1$. 
Denote $H^{q(\cdot)}(\mathbb R^{n})$ the variable Hardy space as in Definition \ref{hx} with $X$ replaced by $L^{q(\cdot)}(\mathbb R^{n})$, which is introduced in \cite{CW, NS}.
For a more detailed discussion on variable Hardy spaces, we refer the reader to
\cite{CMN,Tan,Tan1,TZ,WWX}. 

\begin{theorem}\label{th54} 
Given an integer \( m \geq 1 \), let \( q_1(\cdot), \dots, q_m(\cdot)\in LH \) such that \( 0 < (q_k)_- \leq (q_k)_+ < \infty \). Define
\[
\frac{1}{q(\cdot)} = \frac{1}{q_1(\cdot)} + \cdots + \frac{1}{q_m(\cdot)}.
\]
Let $T$ be a multilinear Calder\'on--Zygmund operator associated to the kernel $K$ that satisfies (\ref{s1e1}) for all \( |\alpha| \leq N \).
If $\frac{mn}{n+N+1}<q^-_k\le q^+_k<\infty$, then $T$ extends to a bounded operator from
$H^{q_1(\cdot)} \times \cdots \times H^{q_m(\cdot)}$ into $L^{q(\cdot)}$.
\end{theorem}

\begin{theorem}\label{th55} 
Given \( q(\cdot), q_1(\cdot), \dots, q_m(\cdot)\) as in Theorem \ref{th54}.
Let $T$ be a multilinear Calder\'on--Zygmund operator associated to the kernel $K$ that satisfies (\ref{s1e1}) for all \( |\alpha| \leq N \).
Assume further that
\begin{align}\label{s1c4}
\int_{(\mathbb R^n)^m}x^\alpha T(a_1,a_2,\cdots,a_m)(x)dx=0,
\end{align}
for all $|\alpha|\le \widetilde N$ and $(L^{q_k(\cdot)},q,N)$-atoms $a_k$,
$k=1,2,\cdots,m$.
If $\frac{mn}{n+N+1}\vee\frac{mn}{N-\widetilde N}<q_k^-\le q_k^+<\infty$
and $\frac{n}{n+\widetilde N+1}<q^-\le q^+<\infty$, then $T$ extends to a bounded operator from
$H^{q_1(\cdot)} \times \cdots \times H^{q_m(\cdot)}$ to $H^{q(\cdot)}$.
\end{theorem}

\begin{theorem}\label{th56}
Let all the notation be as in Theorem \ref{th54}
and let $T$ be a multilinear Calder\'on--Zygmund operator associated to the kernel $K$ that satisfies (\ref{s1e1}).
If $\frac{mn}{n+N+1}<q^-_k\le q^+_k<\infty$, then $T_\ast$ extends to a bounded operator from
$H^{q_1(\cdot)} \times \cdots \times H^{q_m(\cdot)}$ into $L^{q(\cdot)}$.
\end{theorem}

Observe that $h^{q(\cdot)}(\mathbb R^{n})$ is the variable local Hardy space as in Remark \ref{local} with $X$ replaced by $L^{q(\cdot)}(\mathbb R^{n})$. Then we have the following boundedness result.

\begin{theorem}\label{th54-1} 
Let \( q(\cdot), q_1(\cdot), \dots, q_m(\cdot)\) be the same as in Theorem \ref{th54} and let $T$ be a multilinear pseudo-differential operator with the symbol $\sigma\in MB_{1,0}^0$.
If $q_k(\cdot)\in\mathcal{P}^{0}\cap LH$, then $T_\sigma$ extends to a bounded operator from
$h^{q_1(\cdot)} \times \cdots \times h^{q_m(\cdot)}$ into $L^{q(\cdot)}$.
\end{theorem}

\subsection{(Local) Hardy--Morrey spaces}\label{s5s3}
First we recall the definition of Morrey space $M^{p}_{r}(\mathbb{R}^n)$ with $0<r \le p<\infty$. For more details, see \cite{A, FGPW, Mo, SDH, SDH1, YSY} and the references therein.
 For any $x\in \mathbb{R}^n$ and $l\in(0,\infty)$, let $B(x,l):=\{y\in \mathbb{R}^n:|x-y|<l\}$ and $\mathbb{B}(\mathbb{R}^n):=\{B(x,l):x\in\mathbb{R}^n \text{ and } l\in(0,\infty)\}$.
\begin{definition}
	Let $0<r \le p<\infty$. The {\it Morrey space} $M^p_r(\mathbb R^n)$ is defined to be
	the set of all the measurable functions $f$ on $\mathbb R^n$ such that 
	\begin{equation*}
		\|f\|_{M_{r}^{p}\left(\mathbb{R}^{n}\right)}:=\sup _{B \in \mathbb{B}\left(\mathbb{R}^{n}\right)}|B|^{\frac{1}{p}-\frac{1}{r}}\|f\|_{L^{r}(B)}<\infty .
	\end{equation*}
\end{definition}
From the definition of ball-quasi Banach function space \cite[Definition 2.2]{SHYY}, we know that Morrey space $M^p_r(\mathbb R^n)$ is a ball-quasi Banach function space. From \cite[Lemma 2.5]{TX}, we find that Assumption~\ref{ass2.7} holds true for any $s \in (0,1]$ and $\theta\in (0,\min\{s,r\})$.

Let $1<r\le p< \infty$ and $p^{\prime}$, $r^{\prime}$ are conjugate numbers of $p$, $r$. A function $b$ on $\mathbb{R}^n$ is called a $(p^{\prime},r^{\prime})$-block if $\rm{supp}(b)\subset Q $ with $Q\in \mathcal{Q}$, and
\begin{equation}\label{eq7}
	\left(\int_{Q}|b(x)|^{r^{\prime}} d x\right)^{\frac{1}{r^{\prime}}} \leq|Q|^{\frac{1}{p}-\frac{1}{r}} ,
\end{equation}
where $\mathcal{Q}$ denotes the family of all cubes in $\mathbb{R}^n$ with sides parallel to the coordinate axes. The space $\mathcal{B}_{r^{\prime}}^{p^{\prime}}(\mathbb{R}^n)$ is deﬁned by the set of all functions $f$ locally in $L^{r^{\prime}}(\mathbb{R}^n)$ with the norm
\begin{equation*}
	\|f\|_{\mathcal{B}_{r^{\prime}}^{p^{\prime}}\left(\mathbb{R}^{n}\right)}:=\inf \left\{\left\|\left\{\lambda_{k}\right\}\right\|_{l^{1}}: f=\sum_{k} \lambda_{k} b_{k}\right\}<\infty,
\end{equation*}
where $\left\|\left\{\lambda_{k}\right\}_{k=1}^{\infty}\right\|_{l^{1}}=\sum_{k}\left|\lambda_{k}\right|<\infty$ and $b_k$ is a $(p^{\prime},r^{\prime})$-block, and the inﬁmum is taken over all
possible decompositions of $f$. 
Also, for $s\in (0,\min\{1,r\})$ and a $q\in(\max\{1,p\},\infty)$ we have that
\begin{equation*}
	[(M^{p/s}_{r/s})^{\prime}]^{1/(q/s)^{\prime}}(\mathbb{R}^n)=\mathcal{B}^{(p/s)^\prime/(q/s)^\prime}_{(r/s)^\prime/(q/s)^\prime}(\mathbb{R}^n).
\end{equation*}
Hence, from this and the fact that $\mathcal{M}$ is bounded on $\mathcal{B}^{p}_{r}(\mathbb{R}^n)$ for any $1<p\le r< \infty$, we get the 
Assumption~\ref{ass2.8} holds true in this case, which can be found in \cite[Remark 2.7(e)]{WYY}.
Moreover, by the definition of Morrey spaces and
the classical H\"{o}lder inequality on Lebesgue spaces,
we can get the following proposition. For more details, see \cite[Section 11]{SDH1} and \cite[Section 6]{HN}.

\begin{proposition}
Let $0<r_k \le p_k<\infty$, where $k=1,\dots, m$. Define $0<r \le p<\infty$ by
$$\frac{1}{p}=\sum_{k=1}^m\frac{1}{p_k},\quad \frac{1}{r}=\sum_{k=1}^m\frac{1}{r_k}.$$
Then for all $f_k\in {M_{r_k}^{p_k}}(\mathbb R^{n})$ we have
$$\left\|\prod_{k=1}^mf_k\right\|_{M^{p}_r}\lesssim \prod_{k=1}^m\|f_k\|_{M_{r_k}^{p_k}}.$$
\end{proposition}

Therefore, Theorems~\ref{s1th1},\ \ref{s1th2}, \ref{s1th3} and \ref{s1th4} hold true when $X=M^{p}_r(\mathbb R^{n})$, $X_1=M_{r_1}^{p_1}(\mathbb R^{n})$, $\dots,$
$X_m=M_{r_1}^{p_1}(\mathbb R^{n})$ and $p=p_1=\cdots=p_m=1$.
Denote $HM^{p}_r(\mathbb R^{n})$ the Hardy--Morrey space as in Definition \ref{hx} with $X$ replaced by $M^{p}_r(\mathbb R^{n})$, which have studied in \cite{JW, S}.

\begin{theorem}\label{th57} Given an integer \( m \geq 1 \), let \( 0<r\le p<\infty \), \( 0<r_k\le p_k<\infty, \)  such that
\[
\frac{1}{r} = \sum_{k=1}^m\frac{1}{r_k},\quad 
\frac{1}{p} = \sum_{k=1}^m\frac{1}{p_k},
\]
where $k=1,2,\dots,m$.
Let $T$ be a multilinear Calder\'on--Zygmund operator associated to the kernel $K$ that satisfies (\ref{s1e1}) for all \( |\alpha| \leq N \).
If $\frac{mn}{n+N+1}<r_k<\infty$, then $T$ extends to a bounded operator from
$HM^{p_1}_{r_1}\times \cdots \times HM^{p_m}_{r_m}$ into $M^{p}_r$.
\end{theorem}

\begin{theorem}\label{th58} Given \( r, p, r_k, p_k\) as in Theorem \ref{th57}, 
where $k=1,2,\ldots,m$.
Let $T$ be a multilinear Calder\'on--Zygmund operator associated to the kernel $K$ that satisfies (\ref{s1e1}) for all \( |\alpha| \leq N \).
Assume further that
\begin{align}\label{s1c4}
\int_{(\mathbb R^n)^m}x^\alpha T(a_1,a_2,\cdots,a_m)(x)dx=0,
\end{align}
for all $|\alpha|\le \widetilde N$ and $(M^{p_k}_{r_k},q,N)$-atoms $a_k$,
$k=1,2,\cdots,m$.
If $\frac{mn}{n+N+1}\vee\frac{mn}{N-\widetilde N}<r_k<\infty$
and $\frac{n}{n+\widetilde N+1}<r<\infty$, then $T$ extends to a bounded operator from
$HM^{p_1}_{r_1}\times \cdots \times HM^{p_m}_{r_m}$ to $HM^{p}_r$.
\end{theorem}

\begin{theorem}\label{th59}\quad
Let all the notation be as in Theorem \ref{th57}
and let $T$ be a multilinear Calder\'on--Zygmund operator associated to the kernel $K$ that satisfies (\ref{s1e1}).
If $\frac{mn}{n+N+1}<r_k<\infty$, then $T_\ast$ extends to a bounded operator from
$HM^{p_1}_{r_1}\times \cdots \times HM^{p_m}_{r_m}$ into $M^{p}_r$.
\end{theorem}

Denote $hM^{p}_r(\mathbb R^{n})$ the local Hardy--Morrey space as in Remark \ref{local} with $X$ replaced by $M^{p}_r(\mathbb R^{n})$. The boundedness result on multilinear pseudo-differential operators is as follows.

\begin{theorem}\label{th57-1} 
Let \( r, p, r_k, p_k\) be the same as in Theorem \ref{th57}, where $k=1,2,\ldots,m$. 
Let $T$ be a multilinear pseudo-differential operator with the symbol $\sigma\in MB_{1,0}^0$.
If $0<r_k<\infty$, then $T_\sigma$ extends to a bounded operator from
$hM^{p_1}_{r_1}\times \cdots \times hM^{p_m}_{r_m}$ into $M^{p}_r$.
\end{theorem}

\subsection{Mixed-norm (local) Hardy spaces}\label{s5s4}
Now we recall the definition of mixed-norm Lebesgue space $L^{\Vec{p}}(\mathbb R^{n})$. For more information on mixed-norm type spaces, see \cite{BP,HLYY,HY} and the references therein.\par
\begin{definition}
    Let $\Vec{p}:=(p_{1}, \cdots,p_{n})\in(0,\infty]^{n}$. The \emph{mixed-norm Lebesgue space} $L^{\Vec{p}}(\mathbb R^{n})$ is defined to be the set of all the measurable functions $f$ on $\mathbb R^{n}$ such that
    \[
    \|f\|_{L^{\Vec{p}}}:=\left\{  \int_{\mathbb R^{n}}\cdots\left[ \int_{\mathbb R^{n}}\vert f(x_{1},\cdots,x_{n})\vert^{p_{1}}dx_{1}  \right]^{\frac{p_{2}}{p_{1}}}\cdots dx_{n}\right\}^{\frac{1}{p_{n}}}<\infty
    \]
    with the usual modifications made when $p_{i}=\infty$ for some $i\in\{1,\cdots,n\}$. 
\end{definition}
    For any exponent vector $\Vec{p}:=(p_{1}, \cdots,p_{n})$, let
    \[
    p_{-}:=\min\{p_{1},\dots,p_{n}\}\quad{\rm and}\quad p_{+}:=\max\{p_{1},\dots,p_{n}\}.
    \] 
   Similarly, we denote  \[
    p^k_{-}:=\min\{p^k_{1},\dots,p^k_{n}\}\quad{\rm and}\quad p^k_{+}:=\max\{p^k_{1},\dots,p^k_{n}\},\]
    for $k=1,2,\dots,m$.
From the definition of the mixed-norm Lebesgue space, we know that $L^{\Vec{p}}(\mathbb R^{n})$ with $\Vec{p}\in(0,\infty)^{n}$ is a ball quasi-Banach function space. By \cite{HLYY}, Assumption~\ref{ass2.7} holds true when $s\in(0,1]$, $\theta\in(0,\min\{s,p_{-}\})$ and $X:=L^{\Vec{p}}(\mathbb R^{n})$ with $\Vec{p}\in(0,\infty)^{n}$. Meanwhile, the 
Assumption~\ref{ass2.8} holds true when $X:=L^{\Vec{p}}(\mathbb R^{n})$ with $\Vec{p}\in(0,\infty)^{n}$, $s\in(0,p_{-})$ and $q\in(\max\{1,p_{+}\},\infty]$. Furthermore, $L^{\Vec{p}}(\mathbb R^{n})$ has an absolutely continuous quasi-norm. 
Moreover, by the definition of mixed-norm Lebesgue spaces and
the classical H\"{o}lder inequality on Lebesgue spaces,
a simple computation yields the following proposition. Also see \cite[Corollary 5.2]{HN}.

\begin{proposition}\label{mixholder}
Let $0<p_i, p_i^k<\infty$ such that
$$\frac{1}{p_i}=\sum_{k=1}^m\frac{1}{p_i^k},$$
where $i=1,\dots,n $ and $k=1,\dots, m.$
Then for all $f_k\in L^{\Vec{p_k}}(\mathbb R^{n})$ we have
$$\left\|\prod_{k=1}^mf_k\right\|_{L^{\Vec{p}}}\lesssim \prod_{k=1}^m\|f_k\|_{L^{\Vec{p_k}}},$$
where $\Vec{p_k}:=(p^k_{1}, \cdots,p^k_{n})$.
\end{proposition}
 
Therefore, Theorems~\ref{s1th1}, \ref{s1th2}, \ref{s1th3} and \ref{s1th4} hold true when $X= L^{\Vec{p}}(\mathbb R^{n})$, 
$X_1={L^{\Vec{p_1}}}(\mathbb R^{n})$, $\dots,$
$X_m={L^{\Vec{p_m}}}(\mathbb R^{n})$ and $p=p_1=\cdots=p_m=1$.
Denote $H^{\Vec{p}}(\mathbb R^{n})$ the mixed-norm Hardy space as in Definition \ref{hx} with $X$ replaced by $L^{\Vec{p}}(\mathbb R^{n})$, which can be found in \cite{HLYY,HY}.

\begin{theorem}\label{th510} 
Let $0<p_i, p_i^k<\infty$ be the same as in Proposition \ref{mixholder},
where $k=1,2,\ldots,m$. Let $T$ be a multilinear Calder\'on--Zygmund operator associated to the kernel $K$ that satisfies (\ref{s1e1}) for all \( |\alpha| \leq N \).
If $\frac{mn}{n+N+1}<p^k_{-}<\infty$, then $T$ extends to a bounded operator from
$H^{\Vec{p_1}}\times \cdots \times H^{\Vec{p_m}}$ into $L^{\Vec{p}}$.
\end{theorem}

\begin{theorem}\label{th511} Given \( p_i, p_i^k\) as in Theorem \ref{th510}.
Let $T$ be a multilinear Calder\'on--Zygmund operator associated to the kernel $K$ that satisfies (\ref{s1e1}) for all \( |\alpha| \leq N \).
Assume further that
\begin{align}\label{s1c4}
\int_{(\mathbb R^n)^m}x^\alpha T(a_1,a_2,\cdots,a_m)(x)dx=0,
\end{align}
for all $|\alpha|\le \widetilde N$ and $(L^{\Vec{p_k}},q,N)$-atoms $a_k$,
$k=1,2,\cdots,m$.
If $\frac{mn}{n+N+1}\vee\frac{mn}{N-\widetilde N}<p^k_{-}<\infty$
and $\frac{n}{n+\widetilde N+1}<p_{-}<\infty$, then $T$ extends to a bounded operator from
$H^{\Vec{p_1}}\times \cdots \times H^{\Vec{p_m}}$ into $H^{\Vec{p}}$.
\end{theorem}

\begin{theorem}\label{th512}\quad
Let all the notation be as in Theorem \ref{th510}
and let $T$ be a multilinear Calder\'on--Zygmund operator associated to the kernel $K$ that satisfies (\ref{s1e1}).
If $\frac{mn}{n+N+1}<p^k_{-}<\infty$, then $T_\ast$ extends to a bounded operator from
$H^{\Vec{p_1}}\times \cdots \times H^{\Vec{p_m}}$ into $L^{\Vec{p}}$.
\end{theorem}

Noting that $h^{\Vec{p}}(\mathbb R^{n})$ is the local Hardy--Morrey space as in Remark \ref{local} with $X$ replaced by $L^{\Vec{p}}(\mathbb R^{n})$. Then we get the boundedness result on multilinear pseudo-differential operators as follows.

\begin{theorem}\label{th510-1} 
Let $0<p_i, p_i^k<\infty$ be the same as in Proposition \ref{mixholder},
where $k=1,2,\ldots,m$. Let $T_\sigma$ be a multilinear pseudo-differential operator with the symbol $\sigma\in MB_{1,0}^0$.
If $0<p^k_{-}<\infty$, then $T_\sigma$ extends to a bounded operator from
$h^{\Vec{p_1}}\times \cdots \times h^{\Vec{p_m}}$ into $L^{\Vec{p}}$.
\end{theorem}

\subsection{(Local) Hardy--Lorentz spaces}\label{s5s5}
The Lorentz space was first introduced by Lorentz in \cite{Lo}. We refer the reader to \cite{SHYY,Saw,WYY} for more studies on Lorentz spaces.
\begin{definition}
    The \emph{Lorentz space} $L^{p,q}(\mathbb R^{n})$ is defined to be the set of all measurable functions $f$ on $\mathbb R^{n}$ such that, when $p,q\in(0,\infty)$,
    \[
    \|f\|_{L^{p,q}}:=\left\{  \int_{0}^{\infty}\left[ t^{\frac{1}{p}}f^{*}(t) \right]^{q}\frac{dt}{t} \right\}^{\frac{1}{q}}<\infty,
    \]
    and, when $p\in(0,\infty)$ and $q=\infty$,
    \[
    \|f\|_{L^{p,q}}:=\sup_{t\in(0,\infty)}t^{\frac{1}{p}}f^{*}(t)<\infty,
    \]
    where $f^{*}$ denotes the decreasing rearrangement of $f$, which is defined by setting, for any $t\in[0,\infty)$,
    \[
    f^{*}(t):=\inf\{s\in(0,\infty)\colon \mu_{f}(s)\leq t\}
    \]
    with $\mu_{f}(s):=\left\vert  \left\{  x\in\mathbb R^{n}\colon\vert f(x)\vert>s \right\} \right\vert$.
\end{definition}

It is known  in \cite{SHYY} that $L^{p,q}(\mathbb R^{n})$ is a ball Banach function space when $p,q\in(1,\infty)$ or $p\in(1,\infty)$ and $q=\infty$
and is a ball quasi-Banach function space when $p,q\in(0,\infty)$ or $p\in(0,\infty)$ and $q=\infty$. 
From \cite[Theorem 2.3(iii)]{CRS}, Assumption~\ref{ass2.7} holds true when  $s\in(0,1]$ and $\theta\in(0,\min\{s,p,q\})$. 
If $X:=L^{p,r}(\mathbb R^{n})$ with $p\in(0,\infty)$ and $r\in(0,\infty]$. 
By \cite[Theorem 1.4.16]{G} and \cite[Remark 2.7(d)]{WYY},
Assumption~\ref{ass2.8} holds true when $s\in(0,\min\{p,r\})$ and $q\in(\max\{1,p,r\},\infty]$. 

Moreover, we also need the following H\"{o}lder inequality on Lorentz spaces in \cite[Section 4.3.2]{HN} and \cite[Theorem 4.5]{Hu}.

\begin{proposition}\label{mixlorentz}
Let $0<p, q, p_k,q_k<\infty$ such that
$$
\frac{1}{p}=\sum_{k=1}^m\frac{1}{p_k},\quad
\frac{1}{q}=\sum_{k=1}^m\frac{1}{q_k},
$$
where $k=1,\dots, m.$
Then for all $f_k\in L^{p,q}(\mathbb R^{n})$ we have
$$\left\|\prod_{k=1}^mf_k\right\|_{L^{p,q}}\lesssim \prod_{k=1}^m\|f_k\|_{L^{p_k,q_k}}.$$
\end{proposition}
 
Therefore, Theorems~\ref{s1th1}, \ref{s1th2}, \ref{s1th3} and \ref{s1th4} hold true when $X= L^{1,q/p}(\mathbb R^{n})$, 
$X_1={L^{1,q_1/p_1}}(\mathbb R^{n})$, $\dots,$
$X_m={L^{1,q_m/p_m}}(\mathbb R^{n})$.
Denote $H^{p,q}(\mathbb R^{n})$ the Hardy--Lorentz space as in Definition \ref{hx} with $X$ replaced by $L^{p,q}(\mathbb R^{n})$. See \cite{SHYY,WYY}
for more details on the Hardy--Lorentz space.

\begin{theorem}\label{th513} 
Let $p,q, p_k, q_k$ be the same as in Proposition \ref{mixlorentz},
where $k=1,2,\ldots,m$. Let $T$ be a multilinear Calder\'on--Zygmund operator associated to the kernel $K$ that satisfies (\ref{s1e1}) for all \( |\alpha| \leq N \).
If $\frac{mn}{n+N+1}<{p_k\wedge q_k}<\infty$, then $T$ extends to a bounded operator from
$H^{p_1,q_1}\times \cdots \times H^{p_m,q_m}$ into $L^{p,q}$.
\end{theorem}

\begin{theorem}\label{th514} Given \(p, q, p_k, q_k\) as in Theorem \ref{th513}.
Let $T$ be a multilinear Calder\'on--Zygmund operator associated to the kernel $K$ that satisfies (\ref{s1e1}) for all \( |\alpha| \leq N \).
Assume further that
\begin{align}\label{s1c4}
\int_{(\mathbb R^n)^m}x^\alpha T(a_1,a_2,\cdots,a_m)(x)dx=0,
\end{align}
for all $|\alpha|\le \widetilde N$ and $(L^{\Vec{p_k}},q,N)$-atoms $a_k$,
$k=1,2,\cdots,m$.
If $\frac{mn}{n+N+1}\vee\frac{mn}{N-\widetilde N}<{p_k\wedge q_k}<\infty$
and $\frac{n}{n+\widetilde N+1}<p\wedge q<\infty$, then $T$ extends to a bounded operator from
$H^{p_1,q_1}\times \cdots \times H^{p_m,q_m}$ into $H^{p,q}$.
\end{theorem}

\begin{theorem}\label{th515}
Let all the notation be as in Theorem \ref{th513}
and let $T$ be a multilinear Calder\'on--Zygmund operator associated to the kernel $K$ that satisfies (\ref{s1e1}).
If $\frac{mn}{n+N+1}<p\wedge q<\infty$, then $T_\ast$ extends to a bounded operator from
$H^{p_1,q_1}\times \cdots \times H^{p_m,q_m}$ into $L^{p,q}$.
\end{theorem}

We denote by $h^{p,q}(\mathbb R^{n})$ the local Hardy--Morrey space as in Remark \ref{local} with $X$ replaced by $L^{p,q}(\mathbb R^{n})$. Then we have the following boundedness result.

\begin{theorem}\label{th513-1} 
Let $p,q, p_k, q_k$ be the same as in Proposition \ref{mixlorentz},
where $k=1,2,\ldots,m$.  Let $T_\sigma$ be a multilinear pseudo-differential operator with the symbol $\sigma\in MB_{1,0}^0$.
If $0<{p_k\wedge q_k}<\infty$, then $T_\sigma$ extends to a bounded operator from
$h^{p_1,q_1}\times \cdots \times h^{p_m,q_m}$ into $L^{p,q}$.
\end{theorem}

\subsection{(Local) Hardy--Orlicz spaces}\label{s5s6}
A function $\Phi\colon[0,\infty)\to[0,\infty)$ is called an \emph{Orlicz function} if it is non-decreasing and satisfies $\Phi(0)=0$, $\Phi(t)>0$ whenever $t\in(0,\infty)$, and $\lim_{t\to\infty}\Phi(t)=\infty$.
Let $p\in[0,\infty)$ and $\varphi\colon\mathbb R^{n}\times[0,\infty)\to[0,\infty)$ be a function such that, for almost every $x\in\mathbb R^{n}$, $\varphi(x,\cdot)$ is an Orlicz function. The function $\varphi$ is said to be of \emph{uniformly upper} (resp., \emph{lower}) \emph{type} $p\in[0,\infty)$ if there exists a positive constant $C$ such that, for any $x\in\mathbb R^{n}$, $t\in[0,\infty)$ and $s\in[1,\infty)$ (resp.,$s\in[0,1]$), $\varphi(x,st)\leq Cs^{p}\varphi(x,t)$.
The function $\varphi$ is said to satisfy the \emph{uniform Muckenhoupt condition} for some $q\in[1,\infty)$, denoted by $\varphi\in \mathbb A_{q}(\mathbb R^{n})$, if, when $q\in(1,\infty)$,
\[
[\varphi]_{\mathbb A_{q}}:=\sup_{t\in(0,\infty)}\sup_{B\subset\mathbb R^{n}}\frac{1}{\vert B\vert^{q}}\int_{B}\varphi(x,t)dx\left\{ \int_{B}[\varphi(y,t)]^{-\frac{q^{\prime}}{q}}dy \right\}^{\frac{q}{q^{\prime}}}<\infty,
\]
where $1/q+1/q^{\prime}=1$, or
\[
[\varphi]_{\mathbb A_{1}}:=\sup_{t\in(0,\infty)}\sup_{B\subset\mathbb R^{n}}\frac{1}{\vert B\vert}\int_{B}\varphi(x,t)dx\left( \operatorname*{ess\,sup}\limits_{y\in\mathbb B}[\varphi(y,t)]^{-1} \right)<\infty.
\]
The class $\mathbb A_{\infty}(\mathbb R^{n})$ is defined by setting
\[
\mathbb A_{\infty}(\mathbb R^{n}):=\bigcup_{q\in[1,\infty)}\mathbb A_{q}(\mathbb R^{n}).
\]
For any given $\varphi\in\mathbb A_{\infty}(\mathbb R^{n})$, the \emph{critical weight index} $q(\varphi)$ is defined by setting
\[
q(\varphi):=\inf\{q\in[1,\infty)\colon\varphi\in\mathbb A_{q}(\mathbb R^{n})\}.
\]
Then the function $\varphi\colon\mathbb R^{n}\times[0,\infty)\to[0,\infty)$ is called a \emph{growth function} if the following hold true:
\begin{enumerate}
    \item $\varphi$ is a \emph{Musielak--Orlicz function}, namely,\\
    $\varphi(x,\cdot)$ is an Orlicz function for almost every given $x\in\mathbb R^{n}$;\\
    $\varphi(\cdot,t)$ is a measurable function for any given $t\in[0,\infty)$.
    \item $\varphi\in\mathbb A_{\infty}(\mathbb R^{n})$.
    \item The function $\varphi$ is of uniformly lower type $p$ for some $p\in(0,1]$ and of uniformly upper type 1.
\end{enumerate}
\begin{definition}
    For a growth function $\varphi$, a measurable function $f$ on $\mathbb R^{n}$ is said to be in the \emph{Musielak--Orlicz space} $L^{\varphi}(\mathbb R^{n})$ if $\int_{\mathbb R^{n}}\varphi(x,\vert f(x)\vert)dx<\infty$, equipped with the (quasi-)norm
    \[
    \|f\|_{L^{\varphi}}:=\inf\left\{  \lambda\in(0,\infty)\colon\int_{\mathbb R^{n}}\varphi\left( x,\frac{\vert f(x)\vert}{\lambda} \right)dx\leq1\right\}.
    \]
\end{definition}
    
It is shown in \cite[Subsection 7.7]{SHYY} that $L^{\varphi}(\mathbb R^{n})$ is a ball quasi-Banach function space.    
Let $X:=L^{\varphi}(\mathbb R^{n})$ with uniformly lower type $p_{\varphi}^{-}$ and uniformly upper type $p_{\varphi}^{+}$. From \cite[Theorem 7.14(i)]{SHYY}, Assumption~\ref{ass2.7} holds true when $s\in(0,1]$ and $\theta\in(0,\min\{s,p^{-}_{\varphi}/q(\varphi)\})$. Besides,
by \cite[Theorem 7.12]{SHYY} and \cite[Remark 2.7(g)]{WYY},
Assumption~\ref{ass2.8} holds true when $s\in(0,p_\varphi^-)$ and $q\in(1,\infty]$. 
Moreover, we also recall the following H\"{o}lder inequality on Orlicz spaces in \cite[Theorem 2.12]{Wa}.

\begin{proposition}\label{mixorlicz}
Let $L^{\varphi_k}(\mathbb R^{n})$ with uniformly lower type $p_{\varphi_k}^{-}$ and uniformly upper type $p_{\varphi_k}^{+}$,
where $k=1,\dots, m$. Define $\varphi=(\prod_{k=1}^m\varphi_k^{-1})^{-1}$. 
Then for all $f_k\in L^{\varphi_k}(\mathbb R^{n})$, we have
$$\left\|\prod_{k=1}^mf_k\right\|_{L^{\varphi}}\lesssim \prod_{k=1}^m\|f_k\|_{L^{\varphi_k}}.$$
\end{proposition}
 
Therefore, Theorems~\ref{s1th1}, \ref{s1th2}, \ref{s1th3} and \ref{s1th4} hold true when $X= L^{\varphi}(\mathbb R^{n})$, 
$X_1={L^{\varphi_1}}(\mathbb R^{n})$, $\dots,$
$X_m={L^{\varphi_m}}(\mathbb R^{n})$ and $p=p_1=\cdots=p_m=1$.
Denote $H^{\varphi}(\mathbb R^{n})$ the Hardy--Orlicz space as in Definition \ref{hx} with $X$ replaced by $L^{\varphi}(\mathbb R^{n})$. 
For more information on the Hardy--Orlicz space, see \cite{CFYY,LFFY,SHYY,WYY,YLK}.

\begin{theorem}\label{th516} 
Let $L^{\varphi}(\mathbb R^{n}),  L^{\varphi_k}(\mathbb R^{n})$ be the same as in Proposition \ref{mixorlicz},
where $k=1,2,\ldots,m$. Let $T$ be a multilinear Calder\'on--Zygmund operator associated to the kernel $K$ that satisfies (\ref{s1e1}) for all \( |\alpha| \leq N \).
If $\frac{mnq(\varphi_k)}{n+N+1}<{p^{-}_{\varphi_k}}<\infty$, then $T$ extends to a bounded operator from
$H^{\varphi_1}\times \cdots \times H^{\varphi_m}$ into $L^{\varphi}$.
\end{theorem}

\begin{theorem}\label{th517} Let all notation be as in Theorem \ref{th516}.
Let $T$ be a multilinear Calder\'on--Zygmund operator associated to the kernel $K$ that satisfies (\ref{s1e1}) for all \( |\alpha| \leq N \).
Assume further that
\begin{align}\label{s1c4}
\int_{(\mathbb R^n)^m}x^\alpha T(a_1,a_2,\cdots,a_m)(x)dx=0,
\end{align}
for all $|\alpha|\le \widetilde N$ and $(L^{\varphi_k},q,N)$-atoms $a_k$,
$k=1,2,\cdots,m$.
If $\frac{mnq(\varphi_k)}{n+N+1}\vee\frac{mnq(\varphi_k)}{N-\widetilde N}<{p^{-}_{\varphi_k}}<\infty$
and $\frac{nq(\varphi)}{n+\widetilde N+1}<p^{-}_\varphi<\infty$, then $T$ extends to a bounded operator from
$H^{\varphi_1}\times \cdots \times H^{\varphi_m}$ into $H^{\varphi}$.
\end{theorem}

\begin{theorem}\label{th518}
Let all the notation be as in Theorem \ref{th516}
and let $T$ be a multilinear Calder\'on--Zygmund operator associated to the kernel $K$ that satisfies (\ref{s1e1}).
If $\frac{mnq(\varphi_k)}{n+N+1}<{p^{-}_{\varphi_k}}<\infty$, then $T_\ast$ extends to a bounded operator from
$H^{\varphi_1}\times \cdots \times H^{\varphi_m}$ into $L^{\varphi}$.
\end{theorem}

Denote $h^{\varphi}(\mathbb R^{n})$ the local Hardy--Orlicz space as in Remark \ref{local} with $X$ replaced by $L^{\varphi}(\mathbb R^{n})$. Then we get the boundedness result as follows.

\begin{theorem}\label{th516-1} 
Let $L^{\varphi}(\mathbb R^{n}),  L^{\varphi_k}(\mathbb R^{n})$ be the same as in Proposition \ref{mixorlicz},
where $k=1,2,\ldots,m$. Let $T_\sigma$ be a multilinear pseudo-differential operator with the symbol $\sigma\in MB_{1,0}^0$.
If $0<{p^{-}_{\varphi_k}}<\infty$, then $T_\sigma$ extends to a bounded operator from
$h^{\varphi_1}\times \cdots \times h^{\varphi_m}$ into $L^{\varphi}$.
\end{theorem}

\begin{remark}
As mentioned in \cite{HH}, the generalized Orlicz spaces under certain reasonable conditions encompass a wide variety of classical function spaces as special cases, including Lebesgue spaces, weighted Lebesgue spaces, classical Orlicz spaces, variable Lebesgue spaces, and double phase spaces. In this case, the Hardy--Littlewood maximal operator is also bounded on both the generalized Orlicz spaces and their dual spaces. Furthermore, the Fefferman--Stein type vector-valued maximal inequality and the H\"older inequality also hold for the generalized Orlicz spaces. For more details, see \cite{CH, HH, Wa}. Consequently, the boundedness of these multilinear operators can also be established on the generalized (local) Hardy--Orlicz spaces.
 \end{remark}


\section*{Acknowledgments}
The author is supported by the National Natural Science Foundation of China (Grant No.11901309), Natural Science Foundation of Nanjing University of Posts and Telecommunications (Grant No.NY224167) and the Jiangsu Government Scholarship for Overseas Studies.

\bibliographystyle{amsplain}

\end{document}